\documentclass[12pt,reqno]{amsart}
\usepackage{amsmath,amssymb,amsfonts,amsthm}
\usepackage[left=3cm, right=3cm, top=2.5cm, bottom=2.5cm]{geometry}
\usepackage[utf8]{inputenc}
\usepackage[T1]{fontenc}
\usepackage{color}
\usepackage{multirow}% http://ctan.org/pkg/multirow
\usepackage{graphicx}% http://ctan.org/pkg/graphicx
\usepackage{hyperref}
\usepackage{graphicx}% needed for \resizebox
\usepackage{calc}%
\usepackage{hyperref}
\hypersetup{
    colorlinks=true,
    linkcolor=blue,
    filecolor=magenta,
    urlcolor=cyan,
}
\usepackage{lipsum}
\usepackage{blindtext}
\allowdisplaybreaks %diminui espaços desnecessários e quebra equações longas em duas páginas automaticamente

\usepackage{amsaddr}
\newcommand{\capto}{{}^C \! D_t}
\newcommand{\capt}{{}^C \! D^{\alpha}_t}

\newtheorem{theorem}{Theorem}

\newtheorem{proposition}[theorem]{Proposition}

\newtheorem{definition}[theorem]{Definition}

\begin{document}

\title[The Time-Fractional Wave Equation with Acoustic Boundary Conditions]{On the Initial Boundary Value Problem to the Time-Fractional Wave Equation with Acoustic Boundary Conditions}

%%%%%%%%%%%%%%%%%%%%%%%%%%%

\author[P.M. Carvalho-Neto]{\vspace*{-0.5cm} Paulo M. Carvalho-Neto$^1$ \vspace*{-0.5cm}}
\address[Paulo M. de Carvalho Neto]{Department of Mathematics, Federal University of Santa Catarina, Florian\'{o}polis - SC, Brazil.\vspace*{-0.5cm}}
\email[]{$^1$paulo.carvalho@ufsc.br}

\author[C.L. Frota]{C\'{\i}cero L. Frota$^2$ \vspace*{-0.5cm}}
\address[C\'{\i}cero L. Frota]{Department of Mathematics, State University of Maringá, Maring\'{a} - PR, Brazil. \vspace*{-0.5cm}}
\email[]{$^2$clfrota@uem.br}

\author[P.G.P. Torelli]{Pedro G. P. Torelli$^{3,\ast}$ \vspace*{-0.5cm}}
\address[Pedro G. P. Torelli]{Department of Mathematics, State University of Maringá, Maring\'{a} - PR, Brazil.}
\email{$^3$pg54864@uem.br}
\thanks{$^\ast$Corresponding author: pg54864@uem.br}

%%%%%%%%%%%%%%%%%%%%%%%%%%%


%%%%%%%%%%%%%%%%%%%%%%%%%%%


%%%%%%%%%%%%%%%%%%%%%%%%%%%

\begin{abstract}
%%%
This paper is concerned with the study of the well-posedeness for the initial boundary value problem to the time-fractional wave equation with acoustic boundary conditions. The problem is considered in a bounded and connected domain $\Omega \subset {\mathbb{R}^{n}}$, $n \geq 2$, which includes simply connected regions. The boundary of $\Omega$ is made up of two disjoint pieces $\Gamma_{0}$ and $\Gamma_{1}.$ Homogeneous Dirichlet conditions are enforced on $\Gamma_0$, while acoustic boundary conditions are considered on $\Gamma_1$. To establish our main result, we employ the Faedo-Galerkin method and successfully solve a general system of time-fractional ordinary differential equations which extends the scope of the classical Picard-Lindelöf theorem.
\end{abstract}

\maketitle

\section{Introduction} \label{secao1}

Over the past few decades, there has been a growing interest in using fractional calculus in combination with differential equations as a powerful tool for analyzing complex systems. These systems include, among others, diffusion in nerve cells, anomalous diffusion processes in porous media, turbulent fluids, plasma, finance and others; see \cite{CaCa1,CaCaLy1,FeBrSlBaWe1,GoRa1,MeKl1,MuPa1} as a few examples.

In the light of that, in this paper we are particulary interested to address the classical initial boundary value problem (IBVP) for the wave equation with acoustic boundary condition, when we replace the standard time derivative with its natural non-integer generalization,  the Caputo fractional derivative. To be more precise, here we assume that $\Omega\subset\mathbb{R}^n$ (with $n\geq2$) is an open, bounded and connected set, with smooth boundary $\Gamma$ made up of two disjoint parts $\Gamma_0$ and $\Gamma_1$ ($\Gamma=\Gamma_0 \cup \Gamma_1$ and $\Gamma_0 \cap \Gamma_1= \emptyset$), both connected with positive measure and $\nu$ denotes the unit normal vector on $\Gamma_{1}$ pointing towards the exterior of $\Omega$. The main subject of this work is to prove the well posedeness (existence and uniqueness of solution, as well as its countinuous dependence on initial data) to the following IBVP:
\begin{align}
\label{1prob} &\capt u_t(x,t) - \Delta u(x,t) = 0, && (x,t) \in \Omega \times (0,T),\\
\label{2prob} &u(x,t)=0, &&  (x,t) \in \Gamma_0 \times (0,T),\\
\label{3prob} &f(x)\delta_{tt}(x,t) + g(x)\delta_t(x,t) + h(x)\delta(x,t) = -u_t(x,t), &&  (x,t) \in \Gamma_1 \times (0,T), \hspace{1cm}\\
\label{4prob} &\displaystyle \delta_t(x,t)=\frac{\partial u}{\partial \nu}(x,t) , &&  (x,t) \in \Gamma_1 \times (0,T),\\
\label{5prob} &u(x,0)=u_0(x), \quad u_t(x,0)=u_1(x), &&  x \in \Omega,\\
\label{6prob} &\displaystyle \delta(x,0)=\delta_0(x), \quad \delta_t(x,0)= \frac{\partial u_0}{\partial \nu}(x), && x \in \Gamma_1,
\end{align}
where $\capt$ denotes the classical Caputo fractional derivative of order $\alpha\in(0,1]$, $\Delta$ is the Laplacian operator, $f,g,h:\overline{\Gamma_1}\rightarrow\mathbb{R}$ are given functions and finally, $u_0,u_1:\Omega\rightarrow\mathbb{R}$ and $\delta_0:\Gamma_1\rightarrow\mathbb{R}$ are the initial conditions of the system.

In the limit case, when $\alpha=1$, and the acoustic boundary conditions \eqref{3prob} and \eqref{4prob} are imposed on the whole boundary $\Gamma$, we get the problem associated with a wave motion in a fluid
\begin{align*}
 & u_{tt}(x,t) - \Delta u(x,t) = 0, && (x,t) \in \Omega \times (0,T),\\
 &f(x)\delta_{tt}(x,t) + g(x)\delta_t(x,t) + h(x)\delta(x,t) = -u_t(x,t), &&  (x,t) \in \Gamma \times (0,T), \hspace{1cm}\\
 &\displaystyle \delta_t(x,t)=\frac{\partial u}{\partial \nu}(x,t) , &&  (x,t) \in \Gamma \times (0,T),
\end{align*}
introduced by Beale and Rosencrans (\cite{Be1} and \cite{BeRo1}), which gave rise to a big range of more general problems, see for instance, \cite{AlCaCl1,  BrSiClFr1, CaDoCaFrVi1, Fr1, FrCoLa1, FrGo1, FrMeVi1,  GrSaHo1, KoTa1, LiClFrMe1} and references therein.

Still in the context of integer order time derivatives (the classic wave equation), the first paper dealing with a non-linear problem was \cite{FrGo1}, where Frota and Goldstein considered the Carrier non-linear wave equation
\begin{equation*}
u_{tt}(x,t)-M\left(\int_\Omega u^2(x,t)\,dx\right)\Delta u(x,t)+C|u_t(x,t)|^\gamma u_t(x,t)=0,\quad (x,t) \in \Omega \times (0,T],
\end{equation*}
together with \eqref{2prob} - \eqref{6prob}; where $C$ was a nonnegative constant, $M\in C^1([0,\infty);\mathbb{R})$ and $\gamma>0$.

The physical justification for the model can be seen in \cite{Be1}, \cite{BeRo1} and \cite{MoIn1}. Here, just for a brief contextualization, we give some comments. In our context, $\Omega$ represents a region of the space filled with an ideal fluid at rest which is set into motion by sound waves propagating within the domain. Therefore, if $u$ is the potential velocity of the fluid, it satisfies the time-fractional wave equation \eqref{1prob}. The boundary $\Gamma$ is made up two parts $\Gamma_0$ and $\Gamma_1$, with $\Gamma_0$ absorbing (see \eqref{2prob}) and $\Gamma_{1}$ locally reactive, such that each point $x \in \Gamma_1$ responds independently to the pressure caused by the sound waves. This means that $\delta$, the vertical displacement in the normal direction to the boundary $\Gamma_1$ should satisfies the equation \eqref{3prob}. In fact, each point on the boundary $\Gamma_1$ acts like a damped harmonic oscillator that "springs" in response to the sound pressure. Moreover, we also admit that there exists the compatibility between the normal speed of the boundary and the normal speed of the fluid, which is expressed by equation \eqref{4prob}.

Initial value problems for the time fractional wave equation have received extensive coverage in the scientific literature. Often, the Laplace transform has been widely employed as the primary tool for obtaining solutions. For instance, classic works in the field, such as \cite{KiSrTr1} and \cite{Po1}, have extensively discussed Cauchy problems in the context of fractional equations. Additionally, notable contributions, as \cite{Ma1} and \cite{Ma2}, have provided explanations regarding the inherent diffusion-wave phenomena associated with these solutions. The technique of separating variables combined with the Laplace transform has been utilized in \cite{Da-GeJa1} in the study of IBVPs, encompassing both homogeneous and non-homogeneous boundary conditions. More recently, Faedo-Galerkin's method was utilized in \cite{HuYa1} to demonstrate the well-posedness of an IBVP for the time fractional wave equation, albeit in a slightly different sense than the Caputo formulation.

To the best of author's knowledge, this is the first paper considering the time-fractional wave equation coupled with acoustic boundary conditions. In order to facilitate the implementation of numerical methods, as well as to create basis for treating more general nonlinear problems, we apply Faedo-Galerkin's constructive method. It should be mentioned that even in this context of linear equations, when we project our problem into finite-dimensional subspaces by getting the approximated problems, we arrive at a system of time-fractional ordinary differential equations that, as far as we know, has never been treated before.

We drew inspiration from \cite{BeRo1} and \cite{FrGo1} while formulating the class of problem \eqref{1prob}-\eqref{6prob}, where we shall work in a much more general context, namely, that of time-fractional wave equations. Since our approach introduces a more complex problem by incorporating the Caputo fractional time-derivative, some new notions and results should be established. In fact, there are new key challenges when considering problem \eqref{1prob}-\eqref{6prob}, which are successfully addressed in this paper:

\begin{itemize}
\item[(i)] to consider the specificities and restrictions imposed by the Caputo fractional derivative;

\item[(ii)] to establish a more general version of Picard-Lindelöf theorem.
\end{itemize}

The remaining paper is organized as follows. In Section \ref{preliminars} we provide the prerequisites and auxiliary results that are crucial to the development of subsequent sections. Section \ref{picard} is devoted to analyzing a time-fractional ODE system crucial for establishing the initial aspects of our main result. This section also includes observations regarding the system's solution, which are detailed in Subsection \ref{digression}. In Section \ref{existenceresults}, we explore the well-posedness theory of the problem \eqref{1prob} - \eqref{6prob}, and in Section \ref{closingremarks}, we give some concluding remarks.

\section{Notations, Prerequisites and Auxiliary Results} \label{preliminars}

In this section we give the notations for the functional spaces and also introduce the theory of fractional calculus concerning the Caputo fractional derivative. First of all let us set the triple $(\Omega, \Gamma_{0}, \Gamma_{1})$. Throughout the paper $\Omega \subset {\mathbb{R}}^{n} \, (n \geq 2)$ is an open bounded and connected set with smooth boundary $\Gamma$ made up of two disjoint pieces $\Gamma_{0},\, \Gamma_{1}$ both connected with positive measure. Actually $\Gamma_{0}$ and $\Gamma_{1}$ are connected subsets of $\Gamma$ both with positive measure such that $\Gamma = \Gamma_{0} \cup \Gamma_{1}$ and  $\Gamma_{0} \cap \Gamma_{1} = \emptyset.$ We observe that the domains $\Omega$ includes simply connected regions of ${\mathbb{R}}^{n}$.

For the classical functional spaces such as Sobolev spaces and $L^{p}$ spaces  we adopt the standard notation as described in \cite{Lions, LiMa1}. We denote the inner products and norms in $L^2(\Omega)$ and $L^2(\Gamma_1)$ respectively by
\begin{equation*}
(u,v) = \int_\Omega u(x)\, v(x)\, dx, \quad \vert u \vert = \left( \int_\Omega ( u(x) )^2 dx \right)^{\frac{1}{2}}
\end{equation*}
and
\begin{equation*}
(u,v)_{\Gamma_1} = \int_{\Gamma_1} u(x)\, v(x)\, dx, \quad \vert u \vert_{\Gamma_1} = \left( \int_{\Gamma_1} ( u(x) )^2 dx \right)^{\frac{1}{2}}.
\end{equation*}
If $u, v \in H^{1}(\Omega)$, the real Sobolev space of first order, we write
$$
(\nabla u , \nabla v) = \sum_{i=1}^{n} \left( \frac{\partial u}{\partial x_{i}} , \frac{\partial v}{\partial x_{i}}\right) = \sum_{i=1}^{n} \int\limits_{\Omega} \frac{\partial u}{\partial x_{i}}(x) \, \frac{\partial v}{\partial x_{i}} (x) \, dx
$$
\noindent and
$$
\vert \nabla u \vert = \left[ (\nabla u , \nabla u) \right]^{\frac{1}{2}} = \left[ \sum_{i=1}^{n} \int\limits_{\Omega} \left(\frac{\partial u}{\partial x_{i}}(x) \right)^{2} \,  dx \right]^{\frac{1}{2}} .
$$

Let $\mathcal{H}_\Delta(\Omega)=\{u \in H^1(\Omega); \Delta u \in L^2(\Omega)\}$ be the Hilbert endowed with the inner product $(u,v)_{\mathcal{H}_\Delta(\Omega)} = (u,v)_{H^1(\Omega)} + (\Delta u , \Delta v)$.  By $\gamma_0: H^1(\Omega) \longrightarrow H^{\frac{1}{2}}(\Gamma)$ and  $\gamma_1:\mathcal{H}_\Delta(\Omega) \longrightarrow H^{-\frac{1}{2}}(\Gamma)$ we denote the trace map of order zero and the Neumann trace map on $\mathcal{H}_\Delta(\Omega)$ respectively satisfying
$$
\gamma_0(u) = u \vert_{\Gamma} \textrm{ and } \gamma_1(u) = \frac{\partial u}{\partial \nu} \bigg \vert_{\Gamma}, \quad \mbox{for all } u \in \mathcal{D}(\overline{\Omega}).
$$
It is well known that $\gamma_{0}$ and $\gamma_{1}$ are bounded linear operators and for all $u \in \mathcal{H}_\Delta(\Omega)$ and $v \in H^1(\Omega)$ the following generalized Green's formula holds
\begin{equation*}
	(\Delta u , v) + (\nabla u, \nabla v) = \left\langle \gamma_1(u), \gamma_0(v) \right\rangle_{H^{-\frac{1}{2}}(\Gamma) \times H^{\frac{1}{2}}(\Gamma)}\,.
\end{equation*}

Another closed subspace of $H^{1}(\Omega)$ that we address here is the closure of the set $\{ u \in C^{1}(\overline{\Omega}); u = 0 \textrm{ in } \Gamma_{0} \}$ in $H^{1}(\Omega)$, which we denote by $H^1_{\Gamma_0}(\Omega)$. Since $\Omega$ is a regular domain and $\Gamma_{0}$ has positive measure we have that
$$
H^1_{\Gamma_0}(\Omega) = \{ u \in H^1(\Omega); \gamma_0(u) =0 \textrm{ a.e. in } \Gamma_0\}
$$
is a reflexive and separable Hilbert space. Additionally, the Poincaré inequality holds in $H^1_{\Gamma_0}(\Omega)$ and, in view of this inequality, we have that
\begin{equation*}
	((u,v)) = (\nabla u, \nabla v) \quad \textrm{and} \quad \Vert u \Vert = \vert \nabla u \vert,
\end{equation*}
denote an inner product and a norm in $H^1_{\Gamma_0}(\Omega)$ that are equivalent to the usual ones induced by $H^{1}(\Omega)$.

We employ the notations on the Bochner-Lebesgue and Bochner-Sobolev spaces of vector-valued functions, $L^p(0,T;X)$ and $W^{m,p}(0,T;X)$ where $X$ is a Banach space,  as in \cite[Chap. 1]{ArBaHi1}, \cite{Lions} and \cite[Chap. 23]{Ze1}. For $\alpha > 0$ and $f:[0,T]\rightarrow X$, the Riemann-Liouville (RL for short) fractional integral of order $\alpha$ is
\begin{equation*}J_{t}^\alpha f(t):=\dfrac{1}{\Gamma(\alpha)}\displaystyle\int_{0}^{t}{(t-s)^{\alpha-1}f(s)}\,ds,
\end{equation*}
for every $t\in[0,T]$ such that the above integral exists. Above $\Gamma$ is used to denote the classical Euler's gamma function. Also the RL fractional derivative of order $\alpha$ and the Caputo fractional derivative of order $\alpha$ are respectively defined by
\begin{equation*}
	D_{t}^\alpha f(t) := \dfrac{d^{ \lceil \alpha \rceil}}{dt^{ \lceil \alpha \rceil}}\left[J_t^{\lceil \alpha \rceil -\alpha}f(t)\right],
\end{equation*}
\noindent and
\begin{equation*}
	\capt f(t):=D_{t}^\alpha\left[f(t)-\sum_{k=0}^{\lceil \alpha \rceil -1}\dfrac{f^{(k)}(0)}{k!}t^{k}\right],
\end{equation*}
for every $t\in[0,T]$ such that the right side exists. Above we use $\lceil \cdot \rceil $ to represent the ceiling function, i.e., if $m\in\mathbb{N}$ is such that $m-1<\alpha\leq m$, then $\lceil \alpha \rceil=m$.

In \cite{CarFe1} the authors prove that $\{J_t^\alpha:\alpha\geq0\}\subset\mathcal{L}\big(L^p(0,T;X)\big)$ is a $C_0$-semigroup on $L^p(0,T;X), 1 \leq p < \infty$, and $\{J_t^\alpha:\alpha\geq0\}\subset\mathcal{L}\big(C([0,T];X)\big)$ forms a semigroup on $C([0,T];X)$, where $J_t^0f(t)=f(t)$ for almost every $t\in[0,T]$. Concerning the existence of the Caputo fractional derivative $\capt f(t)$ for almost every $t \in [0,T]$, it is enough to consider functions $f \in C^{\lceil\alpha\rceil-1}([0,T];X)$ such that $J_t^{\lceil\alpha\rceil-\alpha}f(t) \, \in W^{\lceil\alpha\rceil,1}(0,T;X)$,  see  Bazhlekova \cite[Section 1.2]{Baz1} and Carvalho-Neto \cite[Section 2.2]{Car1}.

Below we present fundamental results concerning the previously mentioned fractional operators, which are essential for Section \ref{picard} and the estimations in Section \ref{existenceresults}. We emphasize the simplification $'=\frac{d}{dt}$ throughout the remainder of this paper to simplify notation and, more importantly, to avoid excess subscripts on Theorem \ref{maintheorem}.

\begin{proposition} \label{prop1}
Let $0 < \alpha <1$ and $X$ a Banach space.
\begin{itemize}
\item[(i)] Assume that $f \in L^1(0,T)$ is a nonnegative function. Then
\begin{equation} \label{prop07}
J_{t}^{1} f(t) \leq \Big[T^{(1-\alpha)} \Gamma(\alpha)\Big]J_{t}^{\alpha} f(t), \textrm{ for a.e. } t \in [0,T].\vspace*{0.2cm}
\end{equation}

\item[(ii)] If $f \in L^1(0,T;X)$
\begin{equation} \label{id-1}
D^\alpha_t \left[J_{t}^{{\alpha}}f(t)\right]=f(t), \textrm{ for a.e. } t \in [0,T].
\end{equation}
If additionally $J_t^{1-{\alpha}}f \in W^{1,1}(0,T;X)$, then
\begin{equation} \label{id0}J^\alpha_t \big[D^\alpha_t u(t)\big] = u(t) - \frac{t^{\alpha-1}}{\Gamma(\alpha)} \Big( J^{1-\alpha}_s u(s) |_{s=0} \Big)\end{equation}

\item[(iii)] For $f \in C([0,T];X)$,
\begin{equation} \label{id1}
\capt \left[J_{t}^{{\alpha}}f(t)\right]=f(t), \textrm{ for a.e. } t \in [0,T].
\end{equation}
If additionally $J_t^{1-{\alpha}}f \in W^{1,1}(0,T;X)$, then
\begin{equation} \label{id2}
J_{t}^{{\alpha}}\left[ \capt f(t)\right]=f(t)-f(0), \textrm{ for a.e. } t \in [0,T],
\end{equation}
and
\begin{equation} \label{id3}
J_{t}^{{1}}\left[ \capt f(t)\right]=J_t^{1-\alpha}f(t)-f(0)\left[\dfrac{t^{1-\alpha}}{\Gamma(2-\alpha)}\right], \textrm{ for a.e. } t \in [0,T].\vspace*{0.2cm}
\end{equation}

\item[(iv)] If $f \in W^{1,1}(0,T;X)$,
\begin{equation}\label{id4}
\capt f (t) = J^{1-\alpha}_t f'(t), \textrm{ for a.e. } t \in [0,T].
\end{equation}
If additionally $f(0)=0$, then we can reinterpret the equation \eqref{id4} in the form
\begin{equation} \label{id5}
\dfrac{d}{dt}\left[J^{1-\alpha}_t f(t)\right] = J^{1-\alpha}_t f'(t), \textrm{ for a.e. } t \in [0,T].\vspace*{0.2cm}
\end{equation}

\item[(v)] For $f \in C^{1}([0,T];X)$ such that $J_t^{1-\alpha}f(t) \, \in W^{2,1}(0,T;X)$, we have
\begin{equation}\label{id6}
\capto^{\alpha+1} f (t) = \capt f'(t), \textrm{ for a.e. } t \in [0,T].\vspace*{0.2cm}
\end{equation}

\item[(vi)] If $f \in W^{1,2}(0,T;X)$, then
\begin{equation}\label{id7}
\left\|\capt f (t)\right\|_X^2 \leq \left[\frac{T^{1-\alpha}}{\Gamma(2-\alpha)}\right] J^{1-\alpha}_t \left\|f'(t)\right\|_X^2, \textrm{ for a.e. } t \in [0,T].
\end{equation}
\end{itemize}
\end{proposition}
\begin{proof}

\noindent The proof of \eqref{prop07} follows from the definition. For the proof of (\ref{id-1})-(\ref{id5}) we refer the reader to \cite[Proposition 2.35]{Car1} and \cite[Remark 2.10]{CarFe0}. Applying the Leibniz integral rule in conjunction with \eqref{id5}, we have
$$
\capto^{\alpha+1} f (t)=\dfrac{d^2}{dt^2}\Big\{J_t^{1-\alpha}\big[f(t)-f(0)-tf'(0)\big]\Big\}
=\dfrac{d}{dt}\Big\{J_t^{1-\alpha}\big[f^\prime(t)-f^\prime(0)\big]\Big\} =   \capt f^\prime(t),
$$
in $X$, for almost every $t\in[0,T]$, therefore \eqref{id6} holds.

To prove the estimate (\ref{id7}), we observe that \eqref{id4} guarantees
$$
\left\|\capt f(t)\right\|_X^2
 \leq  \left[ J^{1-\alpha}_t \left\|f'(t) \right\|_X \right]^2  =  \left[ \int_0^t \left(\frac{(t-s)^{- \frac{\alpha}{2}}}{\Gamma(1-\alpha)^{\frac{1}{2}}}\right)\left(\frac{(t-s)^{- \frac{\alpha}{2}}}{\Gamma(1-\alpha)^{\frac{1}{2}}}\right) \left\|f'(t)\right\|_X ds\right]^2.
$$
Therefore, Holder's inequality gives
$$
\left\|\capt f(t) \right\|_X^2
\leq \left[\frac{t^{1-\alpha}}{\Gamma(2-\alpha)}\right] J^{1-\alpha}_t \left\|f'(t)\right\|_X^2 \leq  \left[\frac{T^{1-\alpha}}{\Gamma(2-\alpha)}\right] J^{1-\alpha}_t \left\|f'(t)\right\|_X^2,
$$
for almost every $t\in[0,T]$, which completes the proof.
\end{proof}

At this point, we emphasize that the remaining propositions of this section are original (and important) contributions  to the theory, as far as the authors are aware. These propositions are essential for applying the following theorem, originally formulated in \cite[Theorem 4.11]{CarFe0}, which is presented below and used throughout this work.

\begin{theorem}\label{finalcaputo} Assume that $f\in C([0,T])$, $J_t^{1-\alpha}f\in W^{1,1}(0,T)$ and $J_t^{1-\alpha}f^2\in W^{1,1}(0,T)$. Then,
\begin{equation*}
cD_{t}^\alpha\big[f(t)\big]^2\leq2\Big[cD_{t}^\alpha f(t)\Big]f(t),\quad \textrm{for almost every }t\in [0,T].
\end{equation*}
\end{theorem}

The following proposition helps us verify that, in certain circumstances, a function satisfies the hypotheses of Theorem \ref{finalcaputo}.

\begin{proposition} \label{teoju2}
If $u\in C([0,T])$ is such that $\capt u \in C([0,T])$, then $J^{1-\alpha}_t u^2 \in W^{1,1}(0,T)$.
\end{proposition}
\begin{proof} From the continuity of the RL fractional integral (cf. \cite[Theorem 14]{HaLi}) we have that $J^{1-\alpha}_t u\in C([0,T])$. Then, that the identity
$$D_t^\alpha u(t)=\capt u+\dfrac{u(0)t^{-\alpha}}{\Gamma(1-\alpha)},$$
which holds for a.e. $t\in[0,T]$, allows us to deduce that $J_t^{1-\alpha}u\in W^{1,1}(0,T)$. Therefore, if we define $h(t)=\capt u(t)$, we have that \eqref{id2} ensures the identity
\begin{equation*}
u(t)-u(0) = J^\alpha_t h(t),
\end{equation*}
for a.e. $t\in[0,T]$. Consequently,
\begin{multline*}
J^{1-\alpha}_t u^2(t) = J^{1-\alpha}_t \big[ u(0)+J^\alpha_t h(t) \big]^2 \\= \big[u(0)\big]^2 \frac{t^{1-\alpha}}{\Gamma(2-\alpha)} +2u(0)\big[J^1_t h(t)\big] + J^{1-\alpha}_t \big[ J^\alpha_t h(t) \big]^2.
\end{multline*}

It is evident that the first two terms on the right side of the above equality belong to $W^{1,1}(0,T)$. The first follows from direct computation, while the second relies on the continuity of the RL fractional integral and the fact that $h\in C([0,T])$.

To complete the proof, we assert that $J^{1-\alpha}_t \left[ J^\alpha_t h \right]^2\in W^{1,1}(0,T)$. We only need to verify that $D^\alpha_t [J^\alpha_t h]^2\in L^1(0,T)$, since it follows from the continuity of the RL fractional integral of order $1-\alpha$ from $L^1(0,T)$ into $L^1(0,T)$ (cf. \cite[Theorem 4]{HaLi}), that $J^{1-\alpha}_t \left[ J^\alpha_t h \right]^2\in L^1(0,T)$.

Since $h \in C([0,T])$, we have that $J^\alpha_th(t)$ is Hölder continuous with exponent $\alpha$ on $[0,T]$, or simply, $J^\alpha_th \in C^{0,\alpha}([0,T])$ (cf. \cite[Theorem 14]{HaLi}). Thus, we can apply \cite[Lemma 1]{AlAhKi} and \eqref{id-1} to obtain
\begin{eqnarray*}
D^\alpha_t [J^\alpha_t h(t)]^2 = 2 \big[J^\alpha_t h (t)\big]h (t) - \frac{\alpha}{\Gamma(1-\alpha)} \int_0^t \frac{[J^\alpha_t h (t)-J^\alpha_s h (s)]^2}{(t-s)^{\alpha+1}}ds - \frac{[J^\alpha_t h (t)]^2}{\Gamma(1-\alpha)t^\alpha},
\end{eqnarray*}
for a.e. $t\in[0,T]$. We finish this proof by noting that the above equality together with the fact that $J^\alpha_th \in C^{0,\alpha}([0,T])$ is enough for us to deduce that $D^\alpha_t [J^\alpha_t h]^2\in L^1(0,T)$.
\end{proof}

We conclude this section with a result on regularity, crucial for demonstrating $u^{\prime}(0) = u_{1}$ in $H^1_{\Gamma_0}(\Omega)$ during Step 3: Passage to the limit in Section \ref{existenceresults}.
\begin{proposition} \label{teoderivcont}
If $ u \in L^\infty(0,T;X)$ and $D^\alpha_t u \in L^\infty(0,T;X)$ then $u \in C([0,T];X)$ and $u(0)=0$.
\end{proposition}
\begin{proof}
For any $\varepsilon \in [0,\min\{\alpha,1-\alpha\})$, Theorem 3.5 in \cite{CarFe1} ensures that $J^{\alpha-\varepsilon}_t\big[D^\alpha_t u(t)\big]$ and $J^{1-\alpha-\epsilon}_t u(t)$ are in $C([0,T];X)$. Therefore, for $\delta \in (0,\min\{\alpha,1-\alpha\})$ we have
\begin{equation} \label{jduzero}
J^\alpha_t \big[D^\alpha_t u(t)\big] \big|_{t=0} =  J^\delta_t \big\{J^{\alpha-\delta}_t \big[D^\alpha_t u(t)\big]\big\} \big|_{t=0} = 0,
\end{equation}
and
\begin{equation} \label{jduzero1}
J_t^{1-\alpha}u(t) \big|_{t=0} =  J^\delta_t \big[J^{1-\alpha-\delta}_t u(t)\big]\big|_{t=0} = 0.
\end{equation}

Since follows from the hypotheses that $J^{1-\alpha}_t u \in W^{1,1}(0,T;X)$, identities \eqref{id0} and \eqref{jduzero1} ensure that
\begin{equation} \label{ujdu}
u(t) = J^\alpha_t \big[D^\alpha_t u(t)\big] + \frac{t^{\alpha-1}}{\Gamma(\alpha)} \left( J^{1-\alpha}_t u(t) |_{t=0} \right) = J^\alpha_t \big[D^\alpha_t u (t)\big],
\end{equation}
for a.e. $t\in[0,T]$. The proof is now complete, as we observe that $J^\alpha_t \big[D^\alpha_t u\big] \in C([0,T];X)$ and that \eqref{jduzero} together with \eqref{ujdu} guarantees $u(0)=0$.
\end{proof}

\section{A Generalization of Picard-Lindelöf Theorem}\label{picard}

For $\alpha \in (0,1], f:\Omega\times[0,T] \subset \mathbb{R}^{n+1}\rightarrow\mathbb{R}^n$ and $\xi\in \mathbb{R}^{n}$ given, the classical Cauchy problem for the fractional ordinary differential equation in $\mathbb{R}^n$ is the initial value problem
\begin{equation}\label{fracclassical}
\left\{\begin{array}{l}\capt \varphi(t)=f\big(\varphi(t),t\big),\quad\textrm{for }t\in[0,T],\\\varphi(0)=\xi.\end{array}\right.
\end{equation}

The well-posedeness for \eqref{fracclassical} is well-known and has been extensively studied in the literature, for instance see \cite{KiSrTr1,SaKiMa1} as few examples. In this section, we investigate a time-fractional ODE system that generalizes the Cauchy problem \eqref{fracclassical}. Our main goal here is to establish the existence and uniqueness of a solution to this generalized time-fractional ODE system. This result serves as an essential tool for applying the Faedo-Galerkin method in Section \ref{existenceresults}, and notably, to the best of the authors' knowledge, there is currently no formal proof available.

With this in mind,  we consider $\{\alpha_j\}_{j=1}^n\subset(0,1], f:\Omega\times[0,T] \subset \mathbb{R}^{n+1}\rightarrow\mathbb{R}^n$ and $\xi=(\xi_1,\cdots,\xi_n) \in {\mathbb{R}}^{n}$ given and we look for $\varphi = (\varphi_1 , \cdots , \varphi_n): [0,T] \to \mathbb{R}^n$ (the unknown function) satisfying the time-fractional ODE system given by the following set of equations:
\begin{equation}\label{caratheq}
\left\{
\begin{array}{cccl}
\capto^{\alpha_1} \varphi_1(t) &=& f_1\left(\varphi_1(t), \cdots , \varphi_n(t),t\right), & \textrm{for }t\in[0,T],\\
\vdots& & \vdots & \\
\capto^{\alpha_n} \varphi_n(t) &=& f_n\left(\varphi_1(t), \cdots , \varphi_n(t),t\right), & \textrm{for }t\in[0,T],
\end{array}
\right.
\end{equation}
subjected to the initial condition
\begin{equation}\label{carathin}
\varphi(0)= \xi.
\end{equation}

Let us begin by introducing the notion of a solution to the Cauchy problem \eqref{caratheq}-\eqref{carathin}.

\begin{definition}A function $\varphi = (\varphi_1 , \cdots , \varphi_n): [0,T] \to \mathbb{R}^n$ is said to be a solution of the Cauchy problem \eqref{caratheq}-\eqref{carathin} on $[0,T]$ if it satisfies the following conditions:
\begin{itemize}
\item[(i)] $\varphi$ and $t \mapsto (\capto^{\alpha_1}\varphi(t),\ldots,\capto^{\alpha_n}\varphi(t))$ belong to $C([0,T];\mathbb{R}^n)$;
\item[(ii)] $\{\varphi(t): t\in[0,T]\}\subset \Omega$;
\item[(iii)] $\varphi$ satisfies the equations \eqref{caratheq} for all $t\in [0,T]$ and the initial condition \eqref{carathin}.
\end{itemize}
\end{definition}

Now we present an auxiliar result that connects the solution of \eqref{caratheq}-\eqref{carathin} with the solution of an integral equation.

\begin{proposition}\label{integralequ}
Let $f: \Omega\times[0,T] \subset \mathbb{R}^{n+1} \to \mathbb{R}^n$ be a continuous function. Then $\varphi = (\varphi_1 , \cdots , \varphi_n): [0,T] \to \mathbb{R}^n$ is a solution of \eqref{caratheq}-\eqref{carathin} in $[0,T]$ if, and only if, $\varphi \in C([0,T];\mathbb{R}^n)$ and for all $j\in\{1,\cdots,n\}$ the function $\varphi_{j}$ satisfies the integral equation
\begin{equation} \label{eqint}
\varphi_j(t)=\xi_j+\dfrac{1}{\Gamma(\alpha_j)}\int_{0}^t{(t-s)^{\alpha_j-1}f_j\big(\varphi(s),s\big)}ds,\quad\forall t\in [0,T].
\end{equation}
\end{proposition}
\begin{proof}
Assuming that $\varphi$ is a solution of \eqref{caratheq}-\eqref{carathin} in the interval $[0,T]$, we can observe that $t \mapsto J_t^{1-\alpha_j} [\varphi_j(t)-\varphi_j(0)]$ belongs to $C^1([0,T];\mathbb{R})$ for each $1\leq j\leq n$. Consequently, by applying $J^{\alpha_j}_t$ to the $j$-th equation in \eqref{caratheq} and employing (\ref{id2})  we come to
\begin{equation*}
\varphi_j(t) - \varphi_j(0) = J^{\alpha_j}_t f_j(\varphi(t),t), \quad \forall t \in [0,T].
\end{equation*}
This equality and \eqref{carathin} yields \eqref{eqint}.

Conversely assume that $\varphi \in C([0,T];\mathbb{R}^n)$ and \eqref{eqint} holds, for all $1\leq j\leq n$. Then  the continuity of each $t \mapsto f_j(\varphi(t),t)$ gives $t \mapsto J^{\alpha_j}_t f_j(\varphi(t),t)$ belongs to $C([0,T];\mathbb{R})$, for all $1 \leq j \leq n$.  Applying $\capto^{\alpha_j}$ to both sides of \eqref{eqint} and taking into account \eqref{id1} we get
\begin{equation*}
\capto^{\alpha_j} \varphi_j(t) = f_j(\varphi(t),t), \quad \forall t \in [0,T], 1\leq j \leq n,
\end{equation*}
which means that $\varphi$ satisfies the equations \eqref{caratheq} for all $t\in [0,T]$. Finally, to show that $\varphi$ satisfies \eqref{carathin} we observe that
$$
\vert J^{\alpha_j}_t f_j(\varphi(t),t) \vert_{\mathbb{R}} \leq \frac{1}{\Gamma(\alpha_j)} \int_0^t (t-s)^{\alpha_j-1} \|f(\varphi(\cdot),\cdot)\|_{C([0,T];\mathbb{R}^n)} \, ds = C \,  t^{\alpha_j},
$$
where $C= {\|f(\varphi(\cdot),\cdot)\|_{C([0,T];\mathbb{R}^n)}}/{\Gamma(\alpha_j+1)}$. From this inequatily we have
$$
\lim_{t \to 0^{+}} \vert J^{\alpha_j}_t f_j(\varphi(t),t) \vert_{\mathbb{R}} = 0,
$$
and therefore
$$
\lim_{t \to 0^{+}} \dfrac{1}{\Gamma(\alpha_j)}\int_{0}^t{(t-s)^{\alpha_j-1}f_j\big(\varphi(s),s\big)}\,ds = 0\,.
$$
This leads to \eqref{carathin} by taking the limit as $t \to 0^{+}$ in \eqref{eqint}. Whence we have proved that $\varphi$ is a solution of \eqref{caratheq}-\eqref{carathin} in the interval $[0,T]$ and the proof is completed.
\end{proof}

We say that a function $f:\Omega\times[0,T] \subset \mathbb{R}^{n+1} \to \mathbb{R}^n$ is Lipschitz function on the first variable if there exists a constant $\ell >0$ such that
$$
\|f(x,t) - f(y,t) \|_{\mathbb{R}^n} \leq \ell \, \|x-y\|_{\mathbb{R}^n},  \textrm{ for all } t \in [0,T] \textrm{ and } x,y \in \Omega.
$$

Now we face up to the main result of this section:
\begin{theorem}{\label{teo2}}
Let $f:\Omega\times[0,T] \subset \mathbb{R}^{n+1} \to \mathbb{R}^n$ be continuous and Lipschitz function on the first variable. Then for each $\xi\in \Omega$ there exists an unique $\varphi$ solution of the Cauchy problem \eqref{caratheq}-\eqref{carathin} in the interval $[0,T]$.
\end{theorem}
\begin{proof} Without loss of generality we can assume that $\{\alpha_j\}_{j=1}^n\subset(0,1]$ satisfies $ \alpha_1 \leq \alpha_2 \leq \cdots \leq \alpha_n $. We note that the operator $\mathcal{T}:C\left([0,T];\mathbb{R}^n\right) \to C\left([0,T];\mathbb{R}^n\right)$, given by
\begin{equation*}
\mathcal{T}(\varphi(t)) = \xi +
\left(
\begin{array}{c}
J^{\alpha_1}_t f_1(\varphi(t),t)\\
\vdots \\
J^{\alpha_n}_t f_n(\varphi(t),t)
\end{array}
\right), \textrm{ for all } t \in [0,T],
\end{equation*}
is well-defined. Our strategy in proving the theorem is classical, i.e.,  we show that for $m \in \mathbb{N}$, sufficiently large, the operator $\mathcal{T}^m$ is a contraction.

For  $\varphi, \, \psi \in C\left([0,T];\mathbb{R}^n\right)$ we have
\begin{equation*}
\left\| \mathcal{T} \varphi (t) - \mathcal{T} \psi (t) \right\|_{\mathbb{R}^n} \leq \ell\left(\sum_{k=1}^n  J^{\alpha_k}_t\right) \left\| \varphi(t) - \psi(t) \right\|_{\mathbb{R}^n},
\end{equation*}
for every $t\in[0,T]$. Then, by employing the semigroup property of the RL fractional integral we obtain
\begin{equation*}
\left\| \mathcal{T} \varphi (t) - \mathcal{T} \psi (t) \right\|_{\mathbb{R}^n} \leq \ell \left( \sum_{k=1}^n  J^{\alpha_k - \alpha_1}_t \right) J^{\alpha_1}_t  \left\| \varphi(t) - \psi(t) \right\|_{\mathbb{R}^n}, \textrm{ for all } t \in [0,T].
\end{equation*}
The aforementioned estimate enables us to conclude recursively that
\begin{equation*}
\left\| \mathcal{T}^m \varphi (t) - \mathcal{T}^m \psi (t) \right\|_{\mathbb{R}^n} \leq \ell^m \left( \sum_{k=1}^n  J^{\alpha_k-\alpha_1}_t \right)^m J^{m \alpha_1}_t  \left\| \varphi(t) - \psi(t) \right\|_{\mathbb{R}^n},
\end{equation*}
for all $t\in[0,T]$ and $m\in\mathbb{N}$. This inequality yields
\begin{multline}\label{Tn}
\left\| \mathcal{T}^m \varphi - \mathcal{T}^m \psi \right\|_{C([0,T];\mathbb{R}^n)} \\
\leq \ell^m \left\|\sum_{k=1}^n  J^{\alpha_k-\alpha_1}_t \right\|_{\mathcal{L}(C([0,T];\mathbb{R}^n))}^m \left\|J^{m \alpha_1}_t  \left\| \varphi(\cdot) - \psi(\cdot) \right\|_{\mathbb{R}^n}\right\|_{C([0,T];\mathbb{R})}, \textrm{ for all } m \in \mathbb{N} .
\end{multline}

On the other hand, using the properties of the RL fractional integral operator we find
\begin{equation}\label{novaestim1}
\left\Vert \sum_{k=1}^n J^{\alpha_k-\alpha_1}_t\right\Vert_{\mathcal{L}(C([0,T];\mathbb{R}^n))}
 \leq   \sum_{k=1}^n \left\Vert J^{\alpha_k-\alpha_1}_t \right\Vert_{\mathcal{L}(C([0,T];\mathbb{R}^n))}
 \leq  2\sum_{k=1}^{n} T^{\alpha_k-\alpha_1}.
\end{equation}
In order to obtain the above estimate, we have used that $\Gamma(\alpha) \geq {1}/{2}$ for $\alpha > 0.$

Now observe that for all $1 \leq k \leq n$ we have
\begin{equation*}
T^{(\alpha_k-\alpha_1)} \leq
\left \{
\begin{array}{cl}
1, & \textrm{if } 0 < T \leq 1,\\
T^{(\alpha_n-\alpha_1)}, & \textrm{if } T \geq 1,\\
\end{array}
\right.
\end{equation*}
since $ \alpha_1 \leq \alpha_2 \leq \cdots \leq \alpha_n $. From this and \eqref{novaestim1} we obtain the estimate
\begin{equation*}
\left\Vert \sum_{k=1}^n J^{\alpha_k-\alpha_1}_t \right\Vert_{\mathcal{L}(C([0,T];\mathbb{R}^n))} \leq T_M, \textrm{ where } T_M = 2n\max \left\{ T^{(\alpha_n-\alpha_1)} , 1 \right\}.
\end{equation*}

It follows from this inequality and (\ref{Tn}) that
\begin{equation}\label{id8}
\left\Vert \mathcal{T}^m \varphi - \mathcal{T}^m \psi \right\Vert_{C([0,T];\mathbb{R}^n)} \leq \frac{( \ell T_M T^{\alpha_1})^m}{\Gamma(m\alpha_1+1)} \left\Vert \varphi- \psi \right\Vert_{C([0,T];\mathbb{R}^n)}.
\end{equation}
Hence, for sufficiently large $m \in \mathbb{N}$, we have
\begin{equation}\label{id9}
\frac{( \ell \,T_M \,T^{\alpha_1})^m}{\Gamma(m\alpha_1+1)} < 1,
\end{equation}
since $\left[ {( \ell T_M T^{\alpha_1})^m}/{\Gamma(m\alpha_1+1)} \right] \to 0$ as $m \to \infty$, corresponding to the general term of the series defining the Mittag-Leffler function $E_{\alpha_1}(z)$ with $z=\ell\, T_M\, T^{\alpha_1}$.

Using \eqref{id8} and \eqref{id9} we obtain that $\mathcal{T}^m$ is a contraction in $C([0,T];\mathbb{R}^n)$, for such suitable $m$. Therefore, the Banach Fixed Point Theorem ensures the existence of a function $\phi \in C([0,T];\mathbb{R}^n)$ that is the unique fixed point of $\mathcal{T}^m$, and consequently, of $\mathcal{T}$. Thus, we have $\phi(t) = \mathcal{T} \phi(t)$ for every $t\in[0,T]$, which can be expressed as
\begin{equation*}
\phi_j(t)=\xi_j+\dfrac{1}{\Gamma(\alpha_j)}\int_{0}^t{(t-s)^{\alpha_j-1}f_j\big(\phi(s),s\big)} ds,  \textrm{ for all } t \in [0,T] \textrm{ and } 1\leq j \leq n.
\end{equation*}

Finally, Proposition \ref{integralequ} completes the proof.
\end{proof}

%\begin{remark}
%The ordering condition in Theorem \ref{teo2} is an artificial hypothesis introduced for the sake of simplifying the proof. It is worth noting that when this condition is not satisfied, the sequence can be rearranged to fulfill it without impacting the structure of the differential equation or the validity of the proof.
%\end{remark}

\subsection{A Brief Digression} \label{digression}

To conclude this section, we find important to highlight a particular case of Theorem \ref{teo2}. Specifically, we want to address the scenario where $f:\Omega\times[0,T] \subset \mathbb{R}^{n+1} \to \mathbb{R}^n$ is given by $f(x,t)=Ax$, with $A\in M^n(\mathbb{R})$ (the set of square matrices with real entries). In this context, a natural question arises regarding the representation of the matrix function corresponding to the unique solution $\phi(t)$ obtained for the Cauchy problem \eqref{caratheq}-\eqref{carathin} in $[0,\infty)$.

It is well-known in the theory (cf. \cite{Da-GeJa0}) that when $\alpha_1=\alpha_2=\ldots=\alpha_n=\alpha\in(0,1]$, the solution to \eqref{caratheq}-\eqref{carathin} can be represented by the Mittag-Leffler matrix function
\begin{equation}\label{auxsol1}\phi(t)=E_{\alpha}(t^\alpha A)\xi,\quad\forall t\geq0,\end{equation}
where $E_\alpha:\mathbb{C}\rightarrow\mathbb{C}$ is the analytic Mittag-Leffler function. It must be pointed out that the representation \eqref{auxsol1} is consistent with the classical case when $\alpha=1$, which corresponds to the exponential matrix.

We can understand \eqref{auxsol1} as a natural generalization of one dimensional case, where the Mittag-Leffler function ($\alpha\in(0,1)$) and the exponential function ($\alpha=1$) are solutions to their respective Cauchy problems \eqref{caratheq}-\eqref{carathin} in $[0,\infty)$.  However, when considering the Cauchy problem \eqref{caratheq}-\eqref{carathin} with distinct orders of differentiation, it is not valid to assume that they are direct generalizations of the one dimensional case, as such an assumption would not even make sense.

Therefore, to understand and derive the matrix function that satisfies \eqref{caratheq}-\eqref{carathin} with distinct orders of differentiation, it is imperative to comprehend the case $n=2$. With that in mind, let us proceed with this discussion by assuming that $A\in M^2(\mathbb{R})$ is in its Jordan normal form. In other words, we are considering just the matrices
$$A_1=\left[\begin{array}{cc}\lambda&0\\0&\mu\end{array}\right],\quad A_2=\left[\begin{array}{cc}\lambda&1\\0&\lambda\end{array}\right]\quad\textrm{and}\quad A_3=\left[\begin{array}{cc}\lambda&\mu\\-\mu&\lambda\end{array}\right],$$
where $\lambda,\mu\in\mathbb{R}$.

\textit{Case $A_1$:} In this scenario, it becomes apparent that the unique solution $\phi(t)$ to \eqref{caratheq}-\eqref{carathin} can be expressed by
$$\phi(t)=\left(\begin{array}{c}E_{\alpha_1}(t^{\alpha_1}\lambda)\xi_1\\E_{\alpha_2}(t^{\alpha_2}\mu)\xi_2\end{array}\right),$$
for every $t\in[0,\infty)$.\vspace*{0.2cm}

\textit{Case $A_2$:} By applying similar reasoning as before, we can deduce that $\phi_2(t)=E_{\alpha_2}(t^{\alpha_2}\lambda)\xi_2$, for every $t\in[0,\infty)$. Once we have determined $\phi_2(t)$, the first line of the problem becomes a non-homogeneous fractional ordinary differential equation of order $\alpha_1$. Therefore, the fractional version of the variation of constants formula allows us to obtain the expression
$$\phi(t)=\left(\begin{array}{c}E_{\alpha_1}(t^{\alpha_1}\lambda)\xi_1+\displaystyle\int_0^t(t-s)^{\alpha_1-1}E_{\alpha_1,\alpha_1}((t-s)^{\alpha_1}\lambda)E_{\alpha_2}(s^{\alpha_2}\lambda)\xi_2\,ds\\E_{\alpha_2}(t^{\alpha_2}\lambda)\xi_2\end{array}\right),$$
for every $t\in[0,\infty)$.\vspace*{0.2cm}

\textit{Case $A_3$:} This particular case poses the greatest challenge in our analysis. Let us consider the scenario where $\lambda=0$. In this situation, we can employ the equations from \eqref{caratheq}-\eqref{carathin} to demonstrate that $\phi(t)$ satisfies the integral equation
\begin{equation}\label{auxnewpb1}\phi(t)=\xi+\left(\begin{array}{c}\dfrac{t^{\alpha_1}\mu\xi_2}{\Gamma(\alpha_1+1)}\vspace*{0.3cm}
\\\dfrac{t^{\alpha_2}\mu\xi_1}{\Gamma(\alpha_2+1)}\end{array}\right)+J_t^{\alpha_1+\alpha_2}\phi(t),\end{equation}
for every $t\in[0,\infty)$. Computing the solution of the integral equation \eqref{auxnewpb1} is not a straightforward task. However, when $\alpha_1+\alpha_2\leq1$, we can deduce that the solution of \eqref{auxnewpb1} reduces to the solution of the non-homogeneous fractional ordinary differential equation
$$\capto^{\alpha_1+\alpha_2}\phi(t)=-\mu^2\phi(t)+\mu\left(\begin{array}{c}\dfrac{t^{-\alpha_2}\xi_2}{\Gamma(1-\alpha_2)}\vspace*{0.3cm}
\\\dfrac{t^{-\alpha_1}\xi_1}{\Gamma(1-\alpha_1)}\end{array}\right)\quad\textrm{and}\quad\phi(0)=\xi.$$
Using the fractional version of the variation of constants formula, we finally deduce that
\begin{multline*}\phi(t)=E_{\alpha_1+\alpha_2}(-t^{\alpha_1+\alpha_2}\mu^2)\xi\\+\mu\int_0^t(t-s)^{\alpha_1+\alpha_2-1}E_{\alpha_1+\alpha_2,\alpha_1+\alpha_2}(-(t-s)^{\alpha_1+\alpha_2}\mu^2)\left(\begin{array}{c}\dfrac{s^{-\alpha_2}\xi_2}{\Gamma(1-\alpha_2)}\vspace*{0.3cm}\\\dfrac{s^{-\alpha_1}\xi_1}{\Gamma(1-\alpha_1)}\end{array}\right)\,ds,\end{multline*}
for every $t\in[0,\infty)$.

The cases where $\alpha_1+\alpha_2>1$ or $\lambda\neq 0$ pose significant challenges and are omitted here. In this brief section, we simply want to highlight the challenge of obtaining a closed formula for the matrix solution of problem \eqref{caratheq}-\eqref{carathin} when $n=2$. This suggests that the general case of $n>2$ is even more complex and requires a more comprehensive investigation, which is currently lacking in the existing literature and will be conducted in our future research.

\section{Well-Posedness Theory} \label{existenceresults}

This section focuses on the initial boundary value problem \eqref{1prob}-\eqref{6prob}, the central subject of this paper. Here our main goal is to establish the well-posedeness of this problem.
\begin{theorem} \label{maintheorem}
Let $\alpha \in (0,1]$ be a real number and let $f,g,h \in C \left( \overline{\Gamma_1} \right)$ be such that $f$ and $h$ are positive functions and $g$ is a non-negative function. If $u_0 \in H^1_{\Gamma_0}(\Omega) \cap H^2(\Omega)$, $u_1 \in H^1_{\Gamma_0}(\Omega)$ and $\delta_0 \in L^2(\Gamma_1)$, then for all $T >0$ arbitrarily fixed there exists an unique pair of functions $(u,\delta)$, with $u: \Omega \times [0,T] \to \mathbb{R}$ and $\delta : \Gamma_{1}\times [0,T] \to \mathbb{R}$, in the class
\begin{equation}\label{id10}
\begin{array}{l}
 u, u_t \in L^\infty \left( 0,T; H^1_{\Gamma_0}(\Omega) \right), \ \capt u_{t} \in L^\infty \left( 0,T; L^2(\Omega) \right) \\
 \textrm{and }  u(t) \in \mathcal{H}_\Delta(\Omega), \text{ for a.e. } t \in  (0,T);
\end{array}
\end{equation}
\begin{equation}\label{id11}
\delta, \delta_t \in L^\infty \left( 0,T; L^2(\Gamma_1) \right)\textrm{ and } \delta_{tt} \in L^2 \left( 0,T; L^2(\Gamma_1) \right);
\end{equation}
which satisfies
\begin{eqnarray}
&&\capt u_t(t) - \Delta u (t) = 0, \textrm{ in } L^2(\Omega), \text{ for a.e. } t \in (0,T); \label{id12}\\
&&f\delta_{tt}(t) + g\delta_t (t) + h\delta (t) = - \gamma_0( u_t(t)) \textrm{ in } L^2(\Gamma_1), \textrm{ for a.e. } t \in (0,T);\label{id13}\\
&&\int_{\Gamma_1} \delta_t(t) \gamma_0(\varphi) d \Gamma_1 = \left\langle \gamma_1(u(t)) ,\gamma_0(\varphi) \right\rangle_{H^{-\frac{1}{2}}(\Gamma) \times H^{\frac{1}{2}}(\Gamma) },  \forall \varphi \in H^1_{\Gamma_0}(\Omega) \nonumber \\
&& \textrm{and a.e. } t \in (0,T); \label{id14}\\
&& u(0)=u_0 \textrm{ and } u_t(0)=u_1, \textrm{ in } L^2(\Omega); \label{id15}\\
&& \delta(0)=\delta_0  \textrm{ and } \delta_t(0) = \gamma_1(u_0)_{\mid_{\Gamma_{1}}}, \textrm{ in } L^2(\Gamma_1). \label{id16}
\end{eqnarray}
Moreover the pair of functions $(u,\delta)$ depends continuously on the initial data and the parameters $f,g$ and $h$.
\end{theorem}
\begin{proof} We have divided the proof into 5 steps: (1) Approximate problem; (2) A priori estimates; (3) Passage to the limit - Existence; (4) Uniqueness; and (5) Continuous dependence.\vspace*{0.3cm}

\noindent \textbf{Step 1: Approximate problem.}
Let $(w_j)_{j \in \mathbb{N}}$ be an orthonormal basis of $L^2(\Omega)$ given by the eigenfunctions of the operator $-\Delta$ with domain $ H^1_{\Gamma_0}(\Omega) \cap H^2(\Omega)$. Thus $(w_j)_{j \in \mathbb{N}}$ is also a complete orthogonal  system in $H_{\Gamma_0}^1(\Omega)$ and $H^1_{\Gamma_0}(\Omega) \cap H^2(\Omega)$. For each $m \in \mathbb{N}$ we set $W_m =[w_1, \cdots , w_m]$ the linear subspace of $H^1_{\Gamma_0}(\Omega) \cap H^2(\Omega)$ spanned by  $\{w_{1}, \cdots, w_{m}\}$. Similarly, let $(z_j)_{j \in \mathbb{N}}$ be an orthonormal basis of the Hilbert space  $L^2(\Gamma_1)$, constructed such that $\gamma_0(W_m)|_{\Gamma_1} \subset Z_m$ for every $m \in \mathbb{N}$, where $Z_m = [z_{1}, \cdots , z_{m}]$ represents the linear subspace of $L^2(\Gamma_1)$ spanned by  $\{z_{1}, \cdots, z_{m}\}$.

Now, for each $m \in \mathbb{N}$ fixed arbitrary, we seek functions $u_{m}: \Omega \times [0,T] \to \mathbb{R}$ and $\delta_{m}:\Gamma_{1} \times [0,T] \to \mathbb{R}$ in the form
\begin{equation*}
\displaystyle u_m(x,t) = \sum_{k=1}^m a_{km}(t)\, w_k(x) \quad\textrm{and}\quad \delta_m(x,t) = \sum_{k=1}^m b_{km}(t)\,z_k(x)
\end{equation*}
which are solutions to the approximate problem
\begin{align}
& \label{help1}\left( \capt u'_m (t) , w_j \right) + \left( \left( u_m(t),w_j \right) \right) - \left(\delta'_m(t), \gamma_0(w_j) \right)_{\Gamma_1}  = 0, && 1 \leq j \leq m\,;\\
& \label{help2}\left(f \delta''_m (t) + g\delta_m'(t) + h\delta_m(t) , z_j \right)_{\Gamma_1}  + \left( \gamma_0(u'_m(t)),z_j\right)_{\Gamma_1} = 0,  && 1 \leq j \leq m\,; \\
& u_{m}(0) = u_{0m}, \ u'_m(0) = u_{1m}, \ \delta_m(0) = \delta_{0m}, \ \delta'_m(0)= \delta_{1m}, \label{help3}
\end{align}
where
\begin{eqnarray}
&&u_{0m} = \sum_{k=1}^m (u_0,w_k)w_k \to u_{0} \,\, \text{in } H^1_{\Gamma_0}(\Omega) \cap H^2(\Omega), \label{help4}\\
&&u_{1m} = \sum_{k=1}^m (u_1,w_k)w_k \to u_{1}\,\, \text{in } H^1_{\Gamma_0}(\Omega) \label{help5}\\
&&\delta_{0m} = \sum_{k=1}^m (\delta_0,z_k)_{\Gamma_1}z_k \to \delta_{0} \text{ in } L^{2}(\Gamma_{1}) \label{help6}\\
&&\delta_{1m} ={\gamma_1(u_{0m})}_{\mid_{\Gamma_{1}}} =\sum_{k=1}^m c_{km}z_k \to {\gamma_{1}(u_{0})}_{\mid_{\Gamma_{1}}} \text{ in } L^{2}(\Gamma_{1}) , \label{help7}
\end{eqnarray}
for some $c_{km} \in \mathbb{R}$.

By setting $y_{jm}(t)=a'_{jm}(t)$ and $v_{jm}(t)=b'_{jm}(t)$ for each $1 \leq j \leq m$, and reinterpreting the variational identities \eqref{help1}-\eqref{help2} and the initial conditions \eqref{help3}, we obtain an equivalent linear fractional ODE system in $\mathbb{R}^{4m}$:
\begin{equation}\label{caratheq21}
\left\{
\begin{array}{ccl}
\capt y_{m}(t) &=& -A_{1,m} \, a_m(t) + A_{2,m} \, v_m(t),\\
 a'_{m}(t) &=& y_m(t),\\
 v'_{m}(t) &=& -(A_{3,m})^{-1}\big[A_{4,m} \, y_m(t) + A_{5,m} \, v_m(t) + A_{6,m} \, b_m(t)\big],\\
 b'_{m}(t) &=& v_m(t),
\end{array}
\right.
\end{equation}
with the initial conditions
\begin{equation}\label{caratheq22}
\left\{
\begin{array}{ccl}
y_{m}(0) &=& u_{1m}, \\
a_{m}(0) &=& u_{0m}, \\
v_{m}(0) &=& \delta_{1m}, \\
b_{m}(0) &=& \delta_{0m},
\end{array}
\right.
\end{equation}
where
\begin{equation*}
\begin{array}{ll}
y_m(t) = \left(y_{1m}(t), \cdots , y_{mm} (t) \right), & a_m(t) = \left(a_{1m}(t), \cdots , a_{mm}(t) \right), \\
v_m(t) = \left(v_{1m}(t), \cdots , v_{mm}(t) \right), & b_m(t) = \left(b_{1m}(t), \cdots , b_{mm}(t) \right),
\end{array}
\end{equation*}
and
\begin{equation*}
\begin{array}{lll}
 A_{1,m}=\big[((w_j,w_i))\big]_{i,j=1}^m, & A_{2,m}=\big[(z_j,\gamma_0(w_i))_{\Gamma_1}\big]_{i,j=1}^m, & A_{3,m}=\big[(fz_j,z_i)_{\Gamma_1}\big]_{i,j=1}^m, \\
 A_{4,m}=\big[(\gamma_0(w_j),z_i)_{\Gamma_1}\big]_{i,j=1}^m, &  A_{5,m}=\big[(gz_j,z_i)_{\Gamma_1}\big]_{i,j=1}^m, & A_{6,m}=\big[(hz_j,z_i)_{\Gamma_1}\big]_{i,j=1}^m.
\end{array}
\end{equation*}
We emphasize that the continuity and positivity of $f$ ensure $(fz_j,z_i)_{\Gamma_1}>0$ for all $i,j\in\{1,\ldots,m\}$. This, combined with the orthonormality of $(z_j)_{j\in\mathbb{N}}$ in $L^2(\Gamma_1)$, guarantees the invertibility of the matrix $A_{3,m}$.

It is not difficult to notice that the Cauchy problem \eqref{caratheq21}-\eqref{caratheq22} satisfies the assumption of Theorem \ref{teo2}.  Therefore we can deduce the existence and uniqueness of solution for \eqref{caratheq21}-\eqref{caratheq22}. The aforementioned conclusions establish that $y_m$, $\capt y_{m}$, $a^\prime_{m}$, $v^\prime_{m}$, $\capt v^\prime_{m}$ and $b^\prime_{m}$  belong to $C([0,T]; \mathbb{R}^m)$, and therefore
\begin{eqnarray}
&& \nonumber u'_m, \capt u'_m \in C^1 \left([0,T]; H^1_{\Gamma_0}(\Omega) \cap H^2(\Omega) \right)\label{eqap1}, \\
&& \nonumber  \capt\delta_m'' \in C \left([0,T]; L^2(\Gamma_1) \right) \textrm{ and } \delta_m \in C^2 \left([0,T]; L^2(\Gamma_1) \right),\label{eqap11} \\
&&\left( \capt u'_m(t), w \right) + \left( \left( u_m(t),w \right) \right) - \left( \delta'_m(t),\gamma_0 (w) \right)_{\Gamma_1} = 0, \textrm{ for all } w \in W_{m}, \label{eqap2}\\
&&\left(f \delta''_m(t) + g \delta'_m(t) + h \delta_m (t) , z \right)_{\Gamma_1}  + \left( \gamma_0 (u'_m(t)) , z \right)_{\Gamma_1} = 0, \textrm{ for all } z \in Z_{m}, \label{eqap3}
\end{eqnarray}
for each $m\in\mathbb{N}$.\vspace*{0.3cm}

\noindent \textbf{Step 2: A priori estimates. Estimate 1.} Taking $w = 2 u'_m \in W_m$ in \eqref{eqap2} and $z = 2\delta'_m \in Z_m$ in \eqref{eqap3} we deduce
\begin{multline*}
2\left( \capt u'_m(t), u'_m(t) \right) + 2\left( \left( u_m(t), u'_m(t) \right) \right) +2\left(f \delta''_m(t) , \delta'_m(t) \right)_{\Gamma_1}\\
+ 2\left\vert g^{\frac{1}{2}} \delta'_m(t)\right\vert^2_{\Gamma_1} + 2\left(h \delta_m (t),\delta'_m(t) \right)_{\Gamma_1} = 0.
\end{multline*}

To address the first term on the left side of the above identity, we observe that
$$\left( \capt u'_m(t), u'_m(t) \right)=\sum_{j=1}^m\left[\capt a_{jm}'(t)\right]a_{jm}'(t)$$
and
$$\capt \left\vert  u'_m(t) \right\vert^2=\sum_{j=1}^m\capt\left[a_{jm}'(t)\right]^2,$$
since $(w_j)_{j\in\mathbb{N}}$ is an orthonormal basis in $L^2(\Omega)$. Then, supported by Proposition \ref{teoju2}, we apply Theorem \ref{finalcaputo} to establish the inequality
\begin{equation*}
\capt \left\vert  u'_m(t) \right\vert^2 + 2\left\vert g^{\frac{1}{2}} \delta'_m(t)\right\vert_{\Gamma_1}^2 +  \frac{d}{dt} \left[ \left\Vert u_m(t) \right\Vert^2+\left\vert f^{\frac{1}{2}}\delta'_m(t) \right\vert_{\Gamma_1}^2+ \left\vert h^{\frac{1}{2}} \delta_m(t)\right\vert_{\Gamma_1}^2 \right] \leq 0.
\end{equation*}

Taking into account \eqref{id3}, the convergences \eqref{help4}-\eqref{help7} and the continuity of the trace map $\gamma_{1}$  (i.e. $|\gamma_{1}(u)|_{\Gamma_{1}} \leq c_{1} \left\Vert u \right\Vert_{H^1_{\Gamma_0}(\Omega) \cap H^2(\Omega)}$), when applying the integral operator $J^1_t$ to both sides of the last inequality we get
\begin{eqnarray}
&&J^{1-\alpha}_t \left\vert  u'_m(t) \right\vert^2  + \left\Vert u_m(t) \right\Vert^2 + f_0 \left\vert \delta'_m(t) \right\vert_{\Gamma_1}^2 \nonumber\\
&&\leq \frac{T^{1-\alpha}}{\Gamma(2-\alpha)} \left\vert  u'_{m}(0) \right\vert^2 + \left\Vert u_m(0) \right\Vert^2 + f_1\left\vert \delta'_m(0) \right\vert_{\Gamma_1}^2+ h_1 \left\vert \delta_m(0)\right\vert_{\Gamma_1}^2 \nonumber\\
&&=\frac{T^{1-\alpha}}{\Gamma(2-\alpha)} \left\vert  u_{1m} \right\vert^2 + \left\Vert u_{0m} \right\Vert^2 + f_1 \left\vert \gamma_1(u_{0m}) \right\vert_{\Gamma_1}^2+ h_1 \left\vert \delta_{0m}\right\vert_{\Gamma_1}^2\nonumber\\
&&\leq \frac{T^{1-\alpha}}{\Gamma(2-\alpha)} \left\Vert  u_{1m} \right\Vert^2 + (1+f_1 c_1^2)\left\Vert u_{0m} \right\Vert^2_{H^1_{\Gamma_0}(\Omega) \cap H^2(\Omega)} + h_1 \left\vert \delta_{0m} \right\vert^2_{\Gamma_1} \nonumber\\
&&\leq  \frac{T^{1-\alpha}}{\Gamma(2-\alpha)} \left\Vert  u_{1} \right\Vert^2 + (1+f_1 c_1^2)\left\Vert u_{0} \right\Vert^2_{H^1_{\Gamma_0}(\Omega) \cap H^2(\Omega)} + h_1 \left\vert \delta_{0} \right\vert^2_{\Gamma_1},\label{id17}
\end{eqnarray}
where
$$
f_{0} = \min_{x \in \overline{\Gamma_{1}}} f(x), \, f_{1} = \max_{x \in \overline{\Gamma_{1}}} f(x), \textrm{ and } h_{1} = \max_{x \in \overline{\Gamma_{1}}} h(x).
$$

Finally, from \eqref{id7} and \eqref{id17} we can see that there exists a constant $K_{1} > 0$, that depends on  $\alpha, T, f_0, f_1, h_1$ and $c_1$, such that
\begin{multline} \label{aprioriI}
\left\Vert \capt u_m \right\Vert^2_{L^\infty(0,T; L^2(\Omega))}
+ \left\Vert u_m \right\Vert^2_{L^\infty \left(0,T;H^1_{\Gamma_0}(\Omega) \right)}
+ \left\Vert \delta'_m \right\Vert^2_{L^\infty(0,T;L^2(\Gamma_1))}\\
\leq K_1 \left[  \left\Vert u_{0} \right\Vert^2_{H^1_{\Gamma_0}(\Omega) \cap H^2(\Omega)} + \left\Vert u_1 \right\Vert^2 + \left\vert \delta_{0} \right\vert^2_{\Gamma_1} \right],
\end{multline}
which completes estimate 1. Additionally, we can estimate $\delta_m$ in $L^\infty(0,T;L^2(\Gamma_1))$ based on the estimate of $\Vert \delta'_m \Vert_{L^\infty(0,T;L^2(\Gamma_1))}$. In fact, the Fundamental Theorem of Calculus implies
\begin{eqnarray*} \label{aprioriIa}
\nonumber \Vert \delta_m \Vert_{L^\infty(0,T;L^2(\Gamma_1))} & = & \left\Vert J^1_t\left[D^1_t \delta_m\right] + \delta_m(0) \right\Vert_{L^\infty(0,T;L^2(\Gamma_1))} \\
\nonumber & \leq & \Vert J^1_t \delta'_m \Vert_{L^\infty(0,T;L^2(\Gamma_1))} + \Vert \delta_m(0) \Vert_{L^\infty(0,T;L^2(\Gamma_1))} \\
\nonumber & \leq & \frac{T}{\Gamma(2)} \Vert \delta'_m \Vert_{L^\infty(0,T;L^2(\Gamma_1))} + \vert \delta_m(0) \vert_{\Gamma_1} \\
& \leq &  T \Vert \delta'_m \Vert_{L^\infty(0,T;L^2(\Gamma_1))} + \vert \delta_0 \vert_{\Gamma_1}.
\end{eqnarray*}
therefore, \eqref{aprioriI} gives us the desired estimate.\vspace*{0.3cm}

\noindent \textbf{Estimate 2.} We apply $D^1_t=\frac{d}{dt}$ in (\ref{eqap2}) and $\capt$ in (\ref{eqap3}). In the resulting equations we take $w=2\capt u_m' \in W_m$ and $z=2 \delta''_m \in Z_m$, and combine them to obtain
\begin{multline}
2\left( \left[ \capt u'_m(t) \right]', \capt u'_m(t) \right) + 2\left( \left( u'_m(t) , \capt u'_m(t) \right) \right) + 2\left(f \,\capt \delta''_m(t) , \delta''_m(t) \right)_{\Gamma_1}\\
+ 2\left( g \,\capt \delta'_m(t),\delta''_m(t) \right)_{\Gamma_1} + 2\left(h \,\capt \delta_m(t),\delta''_m(t) \right)_{\Gamma_1} = 0. \label{id18}
\end{multline}

Since $\delta'_m \in C^1([0,T];L^2(\Gamma_1))$, making use of \eqref{id1} and \eqref{id4} we can write
\begin{equation*}
\left(g\, \capt \delta'_m(t), \delta''_m(t) \right)_{\Gamma_1} = \left( J^{1-\alpha}_t g^{\frac{1}{2}} \delta''_m(t), \capto^{1-\alpha} J^{1-\alpha}_t g^{\frac{1}{2}} \delta''_m(t) \right)_{\Gamma_1}.
\end{equation*}

As done before in  Estimate 1, considering Proposition \ref{teoju2} and Theorem \ref{finalcaputo}, from \eqref{id18}, the equality above, and Young's inequality, it follows that for any $\varepsilon > 0$,
\begin{eqnarray*}
&& \frac{d}{dt} \left\vert \capt u'_m(t) \right\vert^2 + \capt \left\Vert u'_m(t) \right\Vert^2
+ \capt \left\vert f^{\frac{1}{2}} \delta''_m(t) \right\vert^2_{\Gamma_1}
+ \capto^{1-\alpha} \left\vert J^{1-\alpha}_t g^{\frac{1}{2}} \delta''_m(t)\right\vert_{\Gamma_1}^2 \\
&& \leq  2 \left\vert h\, \capt \delta_m(t) \right\vert_{\Gamma_1} \left\vert \delta''_m(t) \right\vert_{\Gamma_1} \leq  \frac{1}{\varepsilon}\left\vert h \,\capt \delta_m(t) \right\vert_{\Gamma_1}^2 + \varepsilon \left\vert \delta''_m(t) \right\vert_{\Gamma_1}^2 \\
&& \leq \frac{h_1\, T^{1-\alpha}}{\varepsilon \Gamma(2-\alpha)} J^{1-\alpha}_t \vert \delta'_m(t) \vert^2_{\Gamma_1} + \varepsilon \,\left\vert \delta''_m(t) \right\vert_{\Gamma_1}^2 ,
\end{eqnarray*}
where in the last inequality we have used \eqref{id7}.

Integrating this inequality over $(0,t)$ and applying \eqref{id3} it follows that
\begin{eqnarray} \label{priori21}
\nonumber &&  \left\vert  \capt u'_m(t) \right\vert^2 +  J^{1-\alpha}_t \left\Vert u'_m(t) \right\Vert^2  + J^{1-\alpha}_t \left\vert f^{\frac{1}{2}} \delta''_m(t) \right\vert^2_{\Gamma_1} + J^{\alpha}_t \left\vert J^{1-\alpha}_t g^{\frac{1}{2}} \delta''_m(t) \right\vert^2_{\Gamma_1} \\
\nonumber && \leq \frac{h_1 T^{1-\alpha}}{\varepsilon \Gamma(2-\alpha)} J^{2-\alpha}_t \vert \delta'_m(t) \vert^2_{\Gamma_1} + \varepsilon J^1_t \left\vert \delta''_m(t) \right\vert_{\Gamma_1}^2 + \left\vert \capt u'_m(0) \right\vert^2 + \frac{T^{1-\alpha}}{\Gamma(2-\alpha)} \Vert u_m'(0) \Vert^2 \\
&& + \frac{T^{1-\alpha}}{\Gamma(2-\alpha)} \left\vert f^{\frac{1}{2}} \delta_m''(0) \right\vert^2_{\Gamma_1} + \frac{T^\alpha}{\Gamma(1+\alpha)} \left\vert \left( J^{1-\alpha}_t g^{\frac{1}{2}} \delta''_m \right)(0) \right\vert^2_{\Gamma_1}.
\end{eqnarray}

Our job now is to estimate the terms, in the right hand side of the above inequality, involving $\capt u'_m(0)$, $\delta''_m (0)$ and $\left(J^{1-\alpha}_t \delta''_m \right)(0)$. We start noting that the regularity of $\delta_m''$ allows us to affirm that $\left(J^{1-\alpha}_t \delta''_m \right)(0)=0$. Furthermore, the approximated equation \eqref{eqap2} in $t=0$ with $w=\capt u_m'(0)$ gives us
\begin{eqnarray*}
\left\vert \capt u_m'(0) \right\vert^2
 &=& - \left( \left( u_m (0) , \capt u_m'(0) \right) \right) + \left(\delta'_m(t), \gamma_0(\capt u_m'(0)) \right)_{\Gamma_1}\\
&=& \left( \Delta u_m(0) , \capt u'_m(0) \right) \leq \frac{1}{2} \left\vert \Delta u_{0m} \right\vert^2 + \frac{1}{2} \left\vert \capt u_m'(0) \right\vert^2,
\end{eqnarray*}
what gives us the estimate
\begin{equation} \label{priori21a}
\left\vert \capt u_m'(0) \right\vert^2 \leq \left\vert \Delta u_{0m} \right\vert^2 \leq \left\Vert u_0 \right\Vert^2_{H^1_{\Gamma_0}(\Omega) \cap H^2(\Omega)}.
\end{equation}

Taking account the continuity of {$\gamma_{0}$ (i.e. $|\gamma_{0}(u)|_{\Gamma_{1}} \leq c_{0} \left\Vert u \right\Vert_{H^1_{\Gamma_0}(\Omega)} $)} and $\gamma_1$, a similar approach on equation \eqref{eqap3} with $z = f \delta''_m(0)$ allows us deduce that
\begin{eqnarray*}
\left\vert f \delta''_m(0) \right\vert^2_{\Gamma_1}
 &=&  \left( f \delta''_m(0) ,- g \delta'_m(0) -h \delta_m(0) - \gamma_0 ( u'_m(0))\right)_{\Gamma_1} \\
&\leq&  \left\vert f \delta''_m(0) \right\vert_{\Gamma_1} \left\vert - g \gamma_1 (u_{0m}) -h \delta_{0m} - \gamma_0(u_{1m}) \right\vert_{\Gamma_1} \\
&\leq& {\frac{1}{2} \left\vert f \delta''_m(0)\right\vert^2_{\Gamma_1} + (g_1c_1)^2\left\Vert u_{0m} \right\Vert^2 +  2h_1^2\left\vert  \delta_{0m} \right\vert^2_{\Gamma_1} +  2 c_0^2 \left\Vert u_{1m} \right\Vert^2}.
\end{eqnarray*}
Therefore we conclude
\begin{equation}\label{priori22}
\left\vert \delta''_m(0) \right\vert^2_{\Gamma_1} \leq \frac{2(g_1c_1)^2}{f_0} \left\Vert u_{0} \right\Vert^2_{H^1_{\Gamma_0}(\Omega) \cap H^2(\Omega)} +  \frac{4c_0^2}{f_0} \left\Vert u_{1} \right\Vert^2 +  \frac{4h_1^2}{f_0}\left\vert  \delta_{0} \right\vert^2_{\Gamma_1} .
\end{equation}

Thus, with \eqref{priori21a} and \eqref{priori22} in \eqref{priori21}, inequality \eqref{prop07} guarantees that
\begin{multline*}
\left\vert  \capt u'_m(t) \right\vert^2 +  \frac{1}{T^\alpha \Gamma(1-\alpha)} \left[ J^{1}_t \left\Vert u'_m(t) \right\Vert^2  + f_0 J^{1}_t \left\vert \delta''_m(t) \right\vert^2_{\Gamma_1} \right] + \frac{g_0}{T^{1-\alpha} \Gamma(\alpha)} J^{1}_t \left\vert J^{1-\alpha}_t \delta''_m(t) \right\vert^2_{\Gamma_1} \\
\leq \frac{h_1}{\varepsilon } \left[ \frac{T^{1-\alpha}}{\Gamma(2-\alpha)} \right]^2 J^{1}_t \vert \delta'_m(t) \vert^2_{\Gamma_1} + \varepsilon J^1_t \left\vert \delta''_m(t) \right\vert_{\Gamma_1}^2 + C_{2} \left[ \left\Vert u_{0} \right\Vert^2_{H^1_{\Gamma_0}(\Omega) \cap H^2(\Omega)} + \left\Vert u_{1} \right\Vert^2 + \left\vert  \delta_{0} \right\vert^2_{\Gamma_1} \right],
\end{multline*}
where $C_2>0$ is a constant.

Choosing $\varepsilon$ sufficiently small, such that $0<\varepsilon < f_0/T^\alpha \Gamma(1-\alpha)$, for some $C_3>0$, we get
\begin{multline*}
\left\vert  \capt u'_m(t) \right\vert^2 + J^{1}_t \left\Vert u'_m(t) \right\Vert^2  + J^{1}_t \left\vert \delta''_m(t) \right\vert^2_{\Gamma_1} \\
\leq C_3 \left[ \left\Vert \delta'_m \right\Vert^2_{L^2(0,T;L^2(\Gamma_1))} +  \left\Vert u_{0} \right\Vert^2_{H^1_{\Gamma_0}(\Omega) \cap H^2(\Omega)} + \left\Vert u_{1} \right\Vert^2 + \left\vert  \delta_{0} \right\vert^2_{\Gamma_1} \right].
\end{multline*}

From this inequality and \eqref{aprioriI} there exists a constant $K_{2}>0$, that depends on  $\alpha,T,f_0,g_1,h_1,c_0$ and $c_1$, such that
\begin{multline} \label{aprioriII}
 \left\Vert  \capt u'_m \right\Vert^2_{L^\infty(0,T;L^2(\Omega))} + \left\Vert u'_m \right\Vert^2_{L^2(0,T;H^1_{\Gamma_0}(\Omega))}  + \left\Vert \delta''_m \right\Vert^2_{L^2(0,T;L^2(\Gamma_1))} \\
 \leq K_2 \left[ \left\Vert u_{0} \right\Vert^2_{H^1_{\Gamma_0}(\Omega) \cap H^2(\Omega)} + \left\Vert u_{1} \right\Vert^2 + \left\vert  \delta_{0} \right\vert^2_{\Gamma_1} \right].
\end{multline}

Moreover, since there exists $C_4>0$ such that $\vert u \vert \leq C_4 \Vert u \Vert$, using \eqref{id2} and the continuity of the RL fractional integral of order $\alpha$ (cf. \cite[Theorem 3.1]{CarFe1}) we obtain
\begin{eqnarray} \label{aprioriIIa}
\nonumber \Vert u'_m \Vert_{L^\infty(0,T;L^2(\Omega))} & = & \Vert J^\alpha_t\, \capt u'_m + u'_m(0) \Vert_{L^\infty(0,T;L^2(\Omega))} \\
\nonumber & \leq & \Vert J^\alpha_t\, \capt u'_m \Vert_{L^\infty(0,T;L^2(\Omega))} + \Vert u'_m(0) \Vert_{L^\infty(0,T;L^2(\Omega))} \\
\nonumber & \leq & \frac{T^\alpha}{\Gamma(\alpha+1)} \Vert \capt u'_m \Vert_{L^\infty(0,T;L^2(\Omega))} + \vert u'_m(0) \vert \\
& \leq &  \frac{T^\alpha}{\Gamma(\alpha+1)} \Vert \capt u'_m \Vert_{L^\infty(0,T;L^2(\Omega))} + C_4\Vert u_1 \Vert.
\end{eqnarray}
In this manner, \eqref{aprioriII} enables us to estimate $u'_m$ in $L^\infty\left(0,T; L^2(\Omega) \right)$, finishing the a priori estimates. \vspace*{0.3cm}

\noindent \textbf{Step 3: Passage to the limit - Existence.}
The a priori estimates, give us 
\begin{eqnarray*}
&&\left( u_m \right)_{m \in \mathbb{N}} \textrm{ is bounded in } L^\infty \left(0,T;H^1_{\Gamma_0}(\Omega) \right); \\
&&\left( u'_m \right)_{m \in \mathbb{N}} \textrm{ is bounded in } L^\infty \left(0,T;L^{2}(\Omega) \right) \textrm{ and in } L^2 \left(0,T;H^1_{\Gamma_0}(\Omega) \right); \\
&&\left( \capt u'_m \right)_{m \in \mathbb{N}} \textrm{ is bounded in } L^\infty \left(0,T;L^{2}(\Omega) \right);\\
&&\left( \delta_m \right)_{m \in \mathbb{N}} \textrm{ and } \left( \delta'_m \right)_{m \in \mathbb{N}} \textrm{ are bounded in } L^\infty \left(0,T;L^2(\Gamma_1) \right);\\
&&\left( \delta''_m \right)_{m \in \mathbb{N}} \textrm{ is bounded in }L^2 \left(0,T;L^2(\Gamma_1) \right).
\end{eqnarray*}

Therefore, with a standard procedure, we can find subsequences  of $ \left( u_m \right)_{m \in \mathbb{N}} $ and  $\left( \delta_m \right)_{m \in \mathbb{N}}$, for which we still use the same notations, and functions $u$ and $\delta$ such that
\begin{eqnarray}
&&u_m \overset{\star}{\rightharpoonup} u \textrm{ in } L^\infty \left(0,T;H^1_{\Gamma_0}(\Omega) \right), \label{conv1} \\
&&u'_m \overset{\star}{\rightharpoonup} u' \textrm{ in } L^\infty \left(0,T; L^{2}(\Omega) \right) \textrm{ and } u'_m \rightharpoonup u' \textrm{ in } L^2 \left(0,T;H^1_{\Gamma_0}(\Omega) \right), \label{conv2} \\
&&\delta_m \overset{\star}{\rightharpoonup} \delta \, \, \mbox{and} \,\, \delta'_m \overset{\star}{\rightharpoonup} \delta'\,\, \text{ in } L^\infty \left(0,T;L^2(\Gamma_1) \right), \label{conv3}\\
&&\delta''_m \rightharpoonup \delta'' \,\, \mbox{in}\,\, L^{2}(0,T;L^{2}(\Gamma_{1}))\,.\label{conv4}
\end{eqnarray}
\noindent Even more, by the continuity of the application $\gamma_0:H^1_{\Gamma_0}(\Omega) \to L^2(\Gamma_1)$ we have
\begin{equation} \label{fraca7}
 \gamma_0 (u'_m) \rightharpoonup \gamma_0 (u') \text{ in } L^2 \left(0,T;L^2(\Gamma_1) \right).
\end{equation}

Besides the convergences in \eqref{conv1}-\eqref{fraca7}, we must also consider the boundedness of the sequence $\left( \capt u'_m \right)_{m \in \mathbb{N}}$ in $L^\infty \left(0,T;L^{2}(\Omega) \right)$. It's worth noting that, due to the non-standard behavior of the Caputo fractional derivative, this argument is unconventional and requires additional justification. Since $\left( \capt u'_m \right)_{m \in \mathbb{N}}$ is bounded in $L^\infty \left(0,T;L^{2}(\Omega) \right)$, we can extract a subsequence, denoted the same way, and find a function $v$ such that
\begin{equation}\label{caputoconv}
\capt u'_m \overset{\star}{\rightharpoonup} v \,\, \mbox{in}\,\, L^\infty \left(0,T;L^{2}(\Omega) \right).
\end{equation}

The main task now is to prove that we can compute the Caputo fractional derivative of order $\alpha$ of $u'$ and that $v =  \capt u'$. 
%
%To do this, let's first prove that we can compute $u'(0)$.
%%
%\begin{equation}\label{convzero}
%u'_m(0) \to u'(0)\,\, \mbox{in} \,\, L^2(\Omega).
%\end{equation}
%%
%It is clear that, since $u'_m(0) = u_{1m} \to u_{1}$ in $ H^1_{\Gamma_0}(\Omega)$, the validity of \eqref{convzero} depends on establishing that $u$ satisfies the initial condition $u'(0) = u_{1}$ in $ H^1_{\Gamma_0}(\Omega)$.
%
%Let us begin demonstrating that there exists $u'(0)$ and $u'(0) = u_{1}$ in $L^{2}(\Omega)$. 
%
At first, observe that from \eqref{caputoconv} and the definition of Caputo fractional derivative, for all $w \in H^1_{\Gamma_0}(\Omega)$ and $\theta \in \mathcal{D}(0,T)$, we have
\begin{equation*}
\int_0^T \frac{d}{dt} J^{1-\alpha}_t \left( u'_m(t) - u'_m(0), w \right) \theta(t) dt \to \int_0^T \left( v(t) , w \right) \theta(t) dt\,.
\end{equation*}
On the other hand, the continuity of the RL fractional integral, along with \eqref{help5} and \eqref{conv2}, enables us to deduce that
\begin{multline*}
\int_0^T \frac{d}{dt} J^{1-\alpha}_t \left( u'_m(t) - u'_m(0), w \right) \theta(t) dt \\= \int_0^T J^{1-\alpha}_t \left( u'_m(t) - u_{1m}, w \right) \theta'(t)\, dt
\to \int_0^T J^{1-\alpha}_t \left( u'(t) - u_{1}, w \right) \theta'(t)\, dt.
\end{multline*}
Hence, from the uniqueness of limit
\begin{equation*}
\int_0^T \frac{d}{dt} J^{1-\alpha}_t \left( u'(t) - u_1, w \right) \theta(t)\, dt = \int_0^T \left( v(t), w \right) \theta(t)\,dt\,,
\end{equation*}
\noindent for all $w \in H^1_{\Gamma_0}(\Omega)$ and $\theta \in \mathcal{D}(0,T).$ This identity shows that  $D^\alpha_t(u'-u_1)=v$ in $L^\infty(0,T;L^2(\Omega))$. Finally, applying Theorem \ref{teoderivcont} we get that $u^{\prime} \in C([0,T]; L^{2}(\Omega))$ and $u'(0)=u_1$ in $ L^{2}(\Omega)$. From this we may conclude that $\capt u'(t)=v(t)$, for a.e. $t\in(0,T)$, and therefore we have
\begin{equation}\label{caputoconv1}
\capt u'_m \overset{\star}{\rightharpoonup} \capt u' \,\, \mbox{in}\,\, L^\infty \left(0,T;L^{2}(\Omega) \right),
\end{equation}
as we wanted.

%Now, we return to prove that $v = \capt u'$. We set $Q=\Omega \times (0,T)$ and observe that \eqref{conv2} implies
%\begin{equation*}
%u'_m \to u' \quad \textrm{in } \mathcal{D}'(Q),
%\end{equation*}
%\noindent here $\mathcal{D}'(Q)$ is the space of Schwartz distributions on $Q$.
%Thus, the continuity of the operators $D_t=D^1_t$ and $J^{1-\alpha}_t$ in the space $\mathcal{D}'(Q)$ allows us to have
%\begin{equation*} \label{jdujdu'}
%D_t \left\{J^{1-\alpha}_t \left[ u'_m - u'_m(0) \right]\right\} \rightarrow D_t \left\{J^{1-\alpha}_t \left[ u' - u'(0) \right]\right\} \quad \text{in } \mathcal{D}'(Q),
%\end{equation*}
%that means
%\begin{equation*} \label{captucaptu'}
%\capt u'_m \rightarrow \capt u' \quad \text{in } \mathcal{D}'(Q).
%\end{equation*}
%
%Similarly, from \eqref{caputoconv}, we also have
%\begin{equation*}
%\capt u'_m \to v \quad \text{in } \mathcal{D}'(Q)
%\end{equation*}
%and the uniqueness of the limit in the space $\mathcal{D}'(Q)$ implies that $\capt u' = v$ in $\mathcal{D}'(Q)$. Hence,
%\begin{equation*}
%\int_Q \capt u'(x,t)  \varphi(x,t) dx dt =\int_Q v(x,t)  \varphi(x,t) dx dt, \quad \forall \varphi \in \mathcal{D}(Q).
%\end{equation*}
%
%Since $v \in L^\infty(0,T;L^2(\Omega))$, from Du-Bois Raymond Lemma  we conclude that $\capt u'=v$ a.e. in $Q$, and then $\capt u' = v $ in $L^\infty(0,T;L^2(\Omega))$. Thus we can rewrite \eqref{caputoconv} as
%\begin{equation}\label{caputoconv1}
%\capt u'_m \overset{\star}{\rightharpoonup} \capt u' \,\, \mbox{in}\,\, L^\infty \left(0,T;L^{2}(\Omega) \right).
%\end{equation}

Now, we can take the limit in the approximate problem \eqref{eqap2}-\eqref{eqap3}. The approximate equation \eqref{eqap2} and the convergences \eqref{conv1}, \eqref{conv3}, and \eqref{caputoconv1} yield
\begin{equation} \label{acopla1}
\left( \capt u'(t), \varphi \right) + \left( \left( u(t), \varphi \right) \right)  - \left( \delta'(t), \gamma_0 (\varphi) \right)_{\Gamma_1} =0,
\end{equation}
for every $\varphi \in H^1_{\Gamma_0}(\Omega)$ and a.e. $t \in (0,T)$. In particular,
\begin{equation*}
\left\langle - \Delta u(t), \varphi \right\rangle_{\mathcal{D}'(\Omega) \times \mathcal{D}(\Omega)} = \left( \left( u (t) , \varphi \right) \right) = -\left( \capt u'(t), \varphi \right),\, \mbox{for all } \varphi \in \mathcal{D}(\Omega)\,,
\end{equation*}
that implies
\begin{equation} \label{igualdadeqsu}
	- \Delta u(t) = -\capt u'(t) \,\,\, \mbox{in} \, L^{2}(\Omega) \,,\, \, \mbox{for a.e.}\, t \in (0,T),
\end{equation}
since $\capt u'(t) \in L^{2}(\Omega)\,\, \mbox{for a.e.}\, t \in (0,T).$ This proves \eqref{id12} and shows that $u(t)\in \mathcal{H}_{\Delta}(\Omega)$ for a.e. $t \in (0,T).$

Taking into account \eqref{eqap3} and the convergences \eqref{conv3}-\eqref{fraca7} we  conclude
\begin{equation*}
\left(f \delta''(t) + g\delta'(t) +h\delta(t) +\gamma_0 (u'(t)), \psi \right)_{\Gamma_1} =0, \quad \forall \psi \in L^2(\Gamma_1), \text{ a.e. in } (0,T).
\end{equation*}
Consequently, by Du Bois-Raymond Lemma,
\begin{equation} \label{igualdadeqsdelta}
f\delta''(t) + g\delta'(t) +h\delta(t) +\gamma_0 (u'(t)) =0 \text{ in } L^{2}(\Gamma_{1}), \text{ for a.e. } t \in (0,T),
\end{equation}
and \eqref{id13} is proved.

In order to prove \eqref{id14} we multiply (\ref{igualdadeqsu}) by $\varphi \in H^1_{\Gamma_0}(\Omega)$ and integrate over $\Omega$, that leads to
\begin{equation*}
\left( \capt u'(t), \varphi \right) - \left(\Delta u(t), \varphi \right) =0, \quad \text{for a.e. } t \in (0,T).
\end{equation*}
Therefore, the generalized Green identity gives
\begin{equation} \label{acopla2}
\left( \capt u'(t), \varphi \right) + \left( \left( u(t), \varphi \right) \right) - \left\langle \gamma_1(u(t)) , \gamma_0(\varphi) \right\rangle_{H^{-\frac{1}{2}}(\Gamma) \times H^{\frac{1}{2}}(\Gamma)} =0,
\end{equation}
for almost every $t \in (0,T)$ and for every $\varphi \in H^1_{\Gamma_0}(\Omega)$. We compare (\ref{acopla1}) and (\ref{acopla2}) to conclude \eqref{id14}.

To verify that the functions $u$ and $\delta$ satisfy the initial conditions \eqref{id15} and \eqref{id16} we first observe that we have already proven $u^{\prime}(0)=u_{1}$ in $H^{1}_{\Gamma_{0}}(\Omega)$. The regularity of $u$, given by \eqref{id10}, implies $u \in C([0,T]; H^{1}_{\Gamma_{0}}(\Omega))$ and the convergences \eqref {help4}, \eqref{conv1} and \eqref{conv2} yields $u(0)=u_{0}.$ Similarly we have $\delta \in C([0,T]; L^{2}(\Gamma_{1}))$ and from convergences \eqref {help7}, \eqref{conv3} and \eqref{conv4} we obtain $\delta^{\prime}(0) = \gamma_{1}(u_{0})\vert_{\Gamma_{1}}$ in $L^{2}(\Omega)$.
This ends Step 3. \vspace*{0.3cm}

\noindent \textbf{Step 4: Uniqueness.}
Let  $(u,\delta)$ and $(v,\varrho)$ be two pairs of functions satisfying \eqref{id10}-\eqref{id16}. Putting $w=u-v$ and $\zeta = \delta - \varrho$, we observe that the pair $(w,\zeta)$ has the regularity described in \eqref{id10}, \eqref{id11} and satisfies
\begin{eqnarray}
&& \capt w_t(t) - \Delta w(t) = 0\, \,\,\mbox{in} \,\, L^{2}(\Omega)\,, \,\, \mbox{for a.e.}\,\, t \in (0,T); \label{unicidade1} \\
&& f \zeta_{tt}(t) + g \zeta_t(t) + h \zeta(t) = - \gamma_{0}(w_t(t)) \textrm{ in } L^{2}(\Gamma_{1}), \textrm{ for a.e. } t \in (0,T);    \label{unicidade2} \\
&&\int_{\Gamma_1} \zeta_t(t) \gamma_0(\varphi) d \Gamma_1 = \left\langle \gamma_1(w(t)) ,\gamma_0(\varphi) \right\rangle_{H^{-\frac{1}{2}}(\Gamma) \times H^{\frac{1}{2}}(\Gamma) }, \nonumber \\
&& \textrm{ for all } \varphi \in H^1_{\Gamma_0}(\Omega)\,\, \textrm{ and a.e. } \in (0,T); \label{unicidade3}\\
&& w(0) = w_t(0)= 0 \textrm{ in } L^2(\Omega) \textrm{ and }  \zeta(0) = \zeta_t(0) = 0 \textrm{ in } L^2(\Gamma_1). \label{unicidade4}
\end{eqnarray}

Taking the $L^{2}(\Omega)$ inner product, in both sides of \eqref{unicidade1}, by $2 w'(t)$; using generalized Green$^{\prime}$s formula and considering \eqref{unicidade3} we obtain
\begin{equation*}
2 \left( \capt w'(t), w'(t) \right) + 2\left( \left( w(t), w'(t) \right) \right) - 2 \left( \zeta'(t) , \gamma_{0}(w'(t)) \right)_{\Gamma_1} = 0\,.
\end{equation*}
This equality, after taking the $L^{2}(\Gamma_{1})$ inner product, in both sides of \eqref{unicidade2}, by $2\zeta'(t)$, implies that
\begin{multline}
 2\left( \capt w'(t), w'(t) \right) + 2\left( \left( w(t), w'(t) \right) \right) +2\left(f \zeta''(t) , \zeta'(t) \right)_{\Gamma_1} \\
+ 2\left(g \zeta'(t) , \zeta'(t) \right)_{\Gamma_1} + 2\left(h \zeta(t),\zeta'(t) \right)_{\Gamma_1} = 0. \label{unicidade5}
\end{multline}
\noindent From now on, we proceed as in Estimate 1 and obtain
\begin{equation*}
J^{1-\alpha}_t \left\vert  w'(t) \right\vert^2 + \left\Vert w(t) \right\Vert^2+\left\vert \zeta'(t) \right\vert_{\Gamma_1}^2+ \left\vert \zeta(t)\right\vert_{\Gamma_1}^2  \leq 0\,, \mbox{for a.e.} \,\, t \in (0,T)\,,
\end{equation*}
where to arrive at the above inequality, we have considered the null initial conditions \eqref{unicidade4}. Thus we have $w(t) = 0$ in $H^{1}_{\Gamma_{0}}(\Omega)$ and $\zeta(t) =0$ in $L^{2}(\Gamma_{1})$ for a.e. $t \in (0,T)$, which show that the pairs of functions $(u, \delta)$ and $(v, \varrho)$ are equal, concluding the proof of the uniqueness. \vspace*{0.3cm}

\noindent \textbf{Step 5: Continuous dependence.} We now say that the unique pair of function $(u,\delta)$, constructed in the previous steps, constitutes a solution to the problem \eqref{1prob}-\eqref{6prob}. With this step, we conclude the well-posedness of the problem \eqref{1prob}-\eqref{6prob} and demonstrate  that the solution $(u,\delta)$ depends continuously on the parameters $f,\,\,g,\,\,h,\,\,u_0,\,\,u_1$ and $\delta_0.$

Initially we observe that from the weakly lower semicontinuity of the norms and estimates \eqref{aprioriI}, \eqref{aprioriII} and \eqref{aprioriIIa} we get
\begin{eqnarray}\label{semicontinuidade}
&&\left\Vert \capt u' \right\Vert^2_{L^\infty(0,T; L^2(\Omega))}
+ \left\Vert u' \right\Vert^2_{L^\infty(0,T;L^2(\Omega))} + \left\Vert u' \right\Vert^2_{L^2 \left(0,T;H^1_{\Gamma_0}(\Omega) \right)} + \left\Vert u \right\Vert^2_{L^\infty \left(0,T;H^1_{\Gamma_0}(\Omega) \right)} \nonumber\\
&&+ \left\Vert \delta'' \right\Vert^2_{L^2(0,T;L^2(\Gamma_1))}
+  \left\Vert \delta' \right\Vert^2_{L^2(0,T;L^2(\Gamma_1))}   + \left\Vert \delta \right\Vert^2_{L^\infty(0,T;L^2(\Gamma_1))} \nonumber\\
&&\leq K \left[ \left\Vert u_{0} \right\Vert^2_{H^1_{\Gamma_0}(\Omega) \cap H^2(\Omega) } + \left\Vert u_{1} \right\Vert^2 + \left\vert  \delta_{0} \right\vert^2_{\Gamma_1} \right].
\end{eqnarray}

On the other hand, let $(u,\delta)$ and $(\tilde{u},\tilde{\delta})$ be the solutions of the problem \eqref{1prob}-\eqref{6prob} associated to sets of parameters $\{f,g,h,u_0,u_1,\delta_0\}$ and  $\{\tilde{f},\tilde{g},\tilde{h}, \tilde{u}_0,\tilde{u}_1,\tilde{\delta}_0\}$, respectively. We set $v=u-\tilde{u}$ and $\sigma=\delta-\tilde{\delta}$. Therefore, by taking inner products in equations \eqref{igualdadeqsu} and \eqref{igualdadeqsdelta}, as in the previous step, and carefully manipulating terms, we obtain
\begin{multline*}
2 \left( \capt v'(t),v'(t) \right) + 2 \left( \left( v(t),v'(t) \right) \right) + 2 (f \sigma''(t), \sigma'(t) )_{\Gamma_1} + 2 ([f-\tilde{f}] \tilde{\delta}''(t), \sigma'(t) )_{\Gamma_1} \\
 + 2 (g \sigma'(t), \sigma'(t) )_{\Gamma_1} + 2 ([g-\tilde{g}] \tilde{\delta}'(t), \sigma'(t) )_{\Gamma_1} 
 + 2 (h \sigma(t), \sigma'(t) )_{\Gamma_1} + 2 ([h-\tilde{h}] \tilde{\delta}(t), \sigma'(t) )_{\Gamma_1} = 0.
\end{multline*}

Proceeding as in the first a priori estimate, we obtain
\begin{multline*}
\capt \vert v'(t) \vert^2 + \frac{d}{dt} \left[ \Vert v(t) \Vert^2 + \vert f^{\frac{1}{2}} \sigma'(t) \vert^2_{\Gamma_1}  +\vert h^{\frac{1}{2}} \sigma(t) \vert^2_{\Gamma_1} \right] \\
\leq 2 \big [ \Vert f - \tilde{f} \Vert_{C(\overline{\Gamma}_1)} + \Vert g - \tilde{g} \Vert_{C(\overline{\Gamma}_1)} + \Vert h - \tilde{h} \Vert_{C(\overline{\Gamma}_1)} \big ] \big [ \vert \tilde{\delta}''(t) \vert_{\Gamma_1} + \vert \tilde{\delta}'(t) \vert_{\Gamma_1} +  \vert \tilde{\delta}(t) \vert_{\Gamma_1}  \big ]\vert \sigma'(t) \vert_{\Gamma_1} \\
\leq C \big[ \Vert f - \tilde{f} \Vert_{C(\overline{\Gamma}_1)} + \Vert g - \tilde{g} \Vert_{C(\overline{\Gamma}_1)} + \Vert h - \tilde{h} \Vert_{C(\overline{\Gamma}_1)} \big],
\end{multline*}
where in the last inequality we used \eqref{semicontinuidade}.

Integrating over $(0,t)$ and applying \eqref{id3}, we obtain
\begin{multline*}
J^{1-\alpha}_t \vert v'(t) \vert^2 + \Vert v(t) \Vert^2 + \vert \sigma'(t) \vert^2_{\Gamma_1} + \vert \sigma(t) \vert^2_{\Gamma_1} \\
\leq C \left( \vert v'(0) \vert^2 + \Vert v(0) \Vert^2 + \vert \sigma'(0) \vert^2_{\Gamma_1} + \vert \sigma(0) \vert^2_{\Gamma_1} \right. \\
\left. + \Vert f - \tilde{f} \Vert_{C(\overline{\Gamma}_1)} + \Vert g - \tilde{g} \Vert_{C(\overline{\Gamma}_1)} + \Vert h - \tilde{h} \Vert_{C(\overline{\Gamma}_1)} \right).
\end{multline*}

Finally, by \eqref{id7}, taking $t=T$ in the above inequality we conclude
\begin{multline*}
\Vert \capt (u-\tilde{u}) \Vert_{L^\infty(0,T;L^2(\Omega))} + \Vert u-\tilde{u} \Vert^2_{C([0,T];H^1_{\Gamma_0}(\Omega))} + \Vert \delta-\tilde{\delta} \Vert^2_{C^1([0,T];L^2(\Gamma_1))} \\
\leq C \big( \Vert u_0-\tilde{u}_0 \Vert^2_{ H^1_{\Gamma_0}(\Omega)} \cap H^2(\Omega) + \Vert u_1-\tilde{u}_1 \Vert^2 + \vert \delta_0-\tilde{\delta}_0 \vert^2_{\Gamma_1} \\
+ \Vert f - \tilde{f} \Vert_{C(\overline{\Gamma}_1)} + \Vert g - \tilde{g} \Vert_{C(\overline{\Gamma}_1)} + \Vert h - \tilde{h} \Vert_{C(\overline{\Gamma}_1)} \big).
\end{multline*}

\end{proof}

\section{Closing Remarks} \label{closingremarks}

We have not yet discussed in this article the equation used to describe the time-fractional wave equation in \eqref{1prob}. It is important to note that our approach differs from the conventional one found in established literature (e.g., \cite{Ma1} and related references), which is represented by the equation
\begin{equation}
\label{1'}
\tag{1'} \capto^{1+\alpha} u(x,t) - \Delta u(x,t) = 0, \quad \textrm{in} \ \Omega \times (0,T),
\end{equation}
where $0<\alpha<1$.

These two formulations differ fundamentally because, as clarified in item $(v)$ of Proposition \ref{prop1}, the equality $\capto^{1+\alpha} f(t) = \capto^{\alpha} f'(t)$, for a.e. $t\in[0,T]$, is valid only when the function $f$ possesses adequate regularity. In our case, for example, we cannot establish that $J_t^{1-\alpha}u(t) \in W^{2,1}(0,T;L^2(\Omega))$, which prevents us from applying \eqref{id6}.

Nevertheless, the constructions and convergences we have achieved during our study are sufficient to demonstrate that our solution $u$ indeed satisfies \eqref{1'}.

Let us begin by considering the approximated solutions $u_m$ of our main theorem, $\varphi \in H^1_{\Gamma_0}(\Omega) $ and $\theta \in \mathcal{D}(0,T)$. Then we have
\begin{multline*}
\int_0^T \left( \capto^{1+\alpha} u_m(t), \varphi \right) \theta(t) dt
= \int_0^T \frac{d^2}{dt^2} J^{1-\alpha}_t \left( u_m(t) - u_m(0)-tu'_m(0) , \varphi \right) \theta(t) dt \\
= \int_0^T J^{1-\alpha}_t \left( u_m(t) - u_m(0)-tu'_m(0) , \varphi \right) \theta''(t) dt.
\end{multline*}

Additionally, since $u_m$ is sufficiently regular, by \eqref{id6} we have
\begin{equation*}
\int_0^T \left( \capto^{1+\alpha} u_m(t), \varphi \right) \theta(t)dt = \int_0^T \left( \capt u'_m(t), \varphi \right) \theta(t)dt.
\end{equation*}

In this context, we observe from the convergences on Step 3 and the continuity of the RL fractional integral of order $\alpha$ that
\begin{equation*}
\int_0^T \left( \capt u'_m(t), \varphi \right) \theta(t)dt \to \int_0^T \left( \capt u'(t), \varphi \right) \theta(t)dt
\end{equation*}
and
\begin{multline*}
\int_0^T J^{1-\alpha}_t \left( u_m(t) - u_m(0)-tu'_m(0) , \varphi \right) \theta''(t) dt \\
\to \int_0^T J^{1-\alpha}_t \left( u(t) - u(0)-tu'(0) , \varphi \right) \theta''(t) dt.
\end{multline*}

Therefore, the uniqueness of the limit implies that
\begin{equation*}
\capto^{\alpha +1} u(x,t)= \capt u'(x,t),
\end{equation*}
for almost every $(x,t)\in\Omega\times(0,T)$. As a result, we are further investigating, after establishing the appropriate solution concepts, the existence and uniqueness of a strong solution to problem \eqref{2prob}-\eqref{6prob} while considering \eqref{1'} in place of \eqref{1prob}.

\end{document}